\documentclass{amsart}

\newtheorem{theorem}{Theorem}[section]
\newtheorem{lemma}[theorem]{Lemma}
\newtheorem{proposition}[theorem]{Proposition}
\newtheorem{corollary}[theorem]{Corollary}
\newtheorem{definition}[theorem]{Definition}

\theoremstyle{definition}
\newtheorem{remark}[theorem]{Remark}
\newtheorem{example}[theorem]{Example}
\numberwithin{equation}{section}

\usepackage{amsfonts}
\usepackage{amssymb,latexsym}
\usepackage{amsmath,amscd}
\usepackage{stmaryrd}
\usepackage{enumerate}
\usepackage{amssymb}
\usepackage{xspace}
\usepackage[all]{xy}

\newcommand{\R}{\mathbb{R}}

\def\g{{\mathfrak{g}}}              
\def\Mon{\mathcal{N}}               
\DeclareMathOperator\hol{hol}       
\DeclareMathOperator\Hol{Hol}
\DeclareMathOperator\im{Im}         
\DeclareMathOperator\Hom{Hom}       

\def\G{\mathcal{G}}                 
\def\tt{\bold{t}}                   
\def\one{\bold{1}}                  

\newcommand{\vba}{$\mathcal{VB}$-algebroid\xspace }                     
\newcommand{\vbas}{$\mathcal{VB}$-algebroids\xspace}                    
\newcommand{\vbg}{$\mathcal{VB}$-groupoid\xspace }                     
\newcommand{\vbgs}{$\mathcal{VB}$-groupoids\xspace}                    
\newcommand{\set}[1]{\left\{#1\right\}}					

\newcommand{\cc}{\hat{c}}

\newcommand\D{\mathcal{D}}

\newcommand{\DVB}{double vector bundle\xspace}
\newcommand{\DVBs}{double vector bundles\xspace}

\newcommand{\bijar}[1][]{%
 \ar[#1]
 \ar@<0.7ex>@{}[#1]|-*=0[@]{\sim}} 

\DeclareMathOperator{\End}{End}
\DeclareMathOperator{\id}{id}

\DeclareMathOperator{\ad}{ad}

\DeclareMathOperator{\coker}{coker}

\def\E{\mathcal{E}}
\def\D{\mathcal{D}}

\def\br{\begin{remark}}
\def\er{\end{remark}}
\def\Rep{{\bf{\mathcal{R}\mkern-1mu\textbf{\textit{ep}}}_{\boldsymbol{\infty}}^2}}
\def\REP{{\bf{\mathcal{R}\mkern-1mu\textbf{\textit{ep}}}_{\boldsymbol{\infty}}}}
\def\LIE{{\bf\mathcal{L}\mkern-1,5mu\textit{\textbf{ie}}}} 
\def\Lie{{\bf\mathcal{L}\mkern-1,5mu\textit{\textbf{ie}}^{2}_{\boldsymbol{\infty}}}}

\def\VBG{{\bf\mathcal{VB}_{G}}}
\def\VBA{{\bf\mathcal{VB}_{A}}}

\def\sig{\sigma}
\def\VV{\mathsf{VE}}
\def\e{\epsilon}

\def\Dd{D_{A,E,C}}

\def\su{\mathfrak{su}^*_2}
\def\be{\bar{e}}

\def\ale{}

\begin{document}

\title[Obstructions to the integrability of $\mathbf{\mathcal{VB}}$-algebroids]{Obstructions to the integrability of $\mathbf{\mathcal{VB}}$-algebroids}

\author[Cabrera]{Alejandro Cabrera}
\address{Instituto de Matem\'atica, Universidade Federal de Rio de Janeiro  \\
	Av. Athos da Silveira Ramos 149,   Cidade Universit\'aria \\ Ilha do Fund\~ao,  21941-909 Rio de Janeiro - Brasil
	}
\email{acabrera@labma.ufrj.br}
\author[Brahic]{Olivier Brahic}
\address{
    Departamento de Matem\'atica, Universidade Federal do Paran\'a \\
    Jardim das Am\'ericas, 81531-980, Curitiba - Brasil.}
\email{olivier@ufpr.br}
\author[Ortiz]{Cristian Ortiz}
\address{Instituto de Matem\'atica e Estat\'istica, Universidade de S\~ao Paulo\\ 
	Rua do Mat\~ao 1010, Cidade Universit\'aria, 05508-090 S\~ao Paulo - Brasil.}
\email{cortiz@ime.usp.br}

\begin{abstract}
\vbgs are vector bundle objects in the category of Lie groupoids: the total and the base spaces of the vector bundle are Lie groupoids and the vector bundle structure maps are required to define Lie groupoid morphisms. The infinitesimal version of \vbgs are \vbas, namely, vector bundle objects in the category of Lie algebroids.
Following recent developments in the area, we show that a \vba is integrable to a \vbg if and only if its base algebroid is integrable and the spherical periods of certain underlying cohomology classes vanish identically. We illustrate our results in concrete examples. Finally, we obtain as a corollary computable obstructions for a $2$-term representation up to homotopy of Lie algebroid to arise as the infinitesimal counterpart of a smooth such representation of a Lie groupoid.
\end{abstract}

\maketitle

\section{Introduction}

Lie theory for Lie algebroids and Lie groupoids has been extensively studied in recent years in connection with Poisson, symplectic and related geometries \cite{CFpoisson, CFstability, Cprequantization, BCWZ, CZ}. 
A paradigmatic example of the interplay between Lie theory and Poisson/symplectic geometry arises from the correspondence between Poisson manifolds and symplectic groupoids (\cite{CDW}). On the Lie theoretical side, we recall that a Lie groupoid $G\rightrightarrows M$ can be differentiated, giving rise to a Lie algebroid $A=\LIE(G)\to M$, in analogy to Lie groups and Lie algebras. The connection comes from the fact that a Poisson manifold $(M,\pi)$ gives rise to a Lie algebroid $A=T^*_\pi M \to M$ and, if $A=\LIE(G)$, then the (source simply-connected, s.s.c.) integration $G$ carries an additional compatible symplectic structure. The Poisson geometry of $(M,\pi)$ can be studied, to a large extent, through the multiplicative symplectic geometry of $G$ (see e.g. \cite{CFpoisson, CDW}). 

On the other hand, it is well known that, unlike in the case of Lie groups and Lie algebras, not every Lie algebroid can be realized as the infinitesimal counterpart of a Lie groupoid. A Lie algebroid $A$ which is isomorphic to one of the form $\LIE(G)$ for some Lie groupoid $G$ is called \emph{integrable}.
The obstructions appearing in this integrability problem were identified by Crainic and Fernandes in \cite{CF03} in terms of certain monodromy groups $\Mon(A)$ associated to $A$. (These monodromy groups are not easily computable in general.)

In this paper, we are concerned with the integrability of an enriched class of Lie algebroids called \emph{\vbas}. \vbas and \vbgs are examples of \emph{double} structures, namely, they given by squares

 \begin{align*}
\SelectTips{cm}{}
\xymatrix{ D \ar[d]_{p_D}\ar[r]^{p} &A\ar[d]^{p_A}\\
           E \ar[r]^{p_E}& M} \hskip2cm \  \ \xymatrix{ H \ar@<0.5ex>[d]\ar@<-.5ex>[d]\ar[r]^{q_H} &G\ar@<0.5ex>[d]\ar@<-.5ex>[d]\\
           E \ar[r]^{q_E}& M}
\end{align*}
\noindent where the vertical arrows denote Lie algebroid structures (respectively, Lie groupoid structures) and the horizontal arrows denote vector bundle structures. Vertical and horizontal structures  are required to be compatible in the sense that the vertical structure maps must define morphisms for the horizontal structure and vice-versa (see \cite{GM08,GM10,BCdH} for minimalistic characterizations). We shall use  the notation $(D,A,E,M)$ for \vbas and $(H,G,E,M)$ for \vbgs.

The application of the Lie functor to the diagram underlying a \vbg $(H,G,E,M)$ gives rise to a \vba $(\LIE(H),\LIE(G),E,M)$, which is the infinitesimal counterpart of $(H,G,E,M)$ (see e.g. \cite{mackenzie-book}). A \vba isomorphic to the infinitesimal counterpart of a \vbg will be called \emph{integrable}. For example, given a Lie algebroid $A\to M$, its tangent and cotangent bundles can be seen to inherit the structure of \vbas  $(TA,A,TM,M)$ and $(T^*\!A,A,A^*,M)$, respectively. Analogously, given a Lie groupoid $G\rightrightarrows M$ with $\LIE(G)=A$, both its tangent and cotangent bundles define \vbgs 
which integrate the above \vbas (see e.g. \cite{GM08, GM10, mackenzie-book}).

The concept of \vba was introduced by Pradines \cite{Pradines} and it has been further studied by Mackenzie \cite{mackenzie-book}, Gracia-Saz and Mehta \cite{GM08}, among other authors (see  \cite{Mehta, Voronov} for their relation to supergeometry). \vbgs were also studied in \cite{mackenzie-book,GM10}. Lie theory for \vbas and \vbgs was systematically studied in \cite{BCdH}.

\vbas and \vbgs have shown to be specially important in the infinitesimal description of Lie groupoids equipped with multiplicative geometric structures, namely, in the study of generalizations of the symplectic/Poisson  correspondence mentioned earlier. For example, multiplicative differential forms turn out to be equivalent to cocycles involving the \vbg structure on $TG$ while multiplicative multivector fields to cocycles involving $T^*\!G$. Their infinitesimal counterpart then becomes clear: they correspond to cocycles for \vba structures arising from $TA$ and $T^*\!A$ (see \cite{BC,MX2} and references therein). More general \vbgs and \vbas appear when considering other multiplicative geometries like multiplicative foliations \cite{Jotz12, JotzOrtiz, DJO} or, more generally, multiplicative Dirac structures \cite{LiblandThesis,Ortiz13}. 
  The $\mathcal{VB}$-structure plays a fundamental role in all these cases and hints that Lie theory in the $\mathcal{VB}$-category is an interesting object to study within this area.



The aim of this work is to characterize the obstructions appearing in the integrability problem for \vbas, namely, the obstructions for the existence of a \vbg $(H,G,E,M)$ such that $\LIE(H,G,E,M)$ is isomorphic to a given \vba. A key result on that matter was shown in \cite{BCdH}, stating that a \vba $(D,A,E,M)$ is integrable iff $D\to E$ is integrable as a Lie algebroid. This means that if the total space $D\to E$ of a \vba is integrable then so is $A\to M$ and, furthermore, the extra vector bundle structure integrates automatically at the level of the (s.s.c.) Lie groupoids. For this reason we only need to study the integrability of $D$. 

Our main result is given by Theorem \ref{thm:integrals}, which states that a \vba $(D,A,E,M)$ is integrable iff $A$ is integrable and the spherical periods of certain cohomology classes associated to $(D,A,E,M)$ vanish. Assuming that $A$ is integrable, we first show that the obstructions to the integrability lie in $\Mon(D)\cap \ker(p:D\to A)$ and then express the elements in this set as spherical integrals of an underlying $2$-cochain on $A$. The first part comes from the general theory of obstructions of \cite{CF03} adapted to the $\mathcal{VB}$-case and the second from the description of the structure of a \vba in terms of connections and cochains given in \cite{GM08}. The vanishing of these integrals can also be seen as the condition for certain algebroid cohomology classes to be in the image of the Van Est map (\cite{C00}).

We end this introduction with a second 'extrinsic' approach to \vbas and \vbgs. Here, the key point stems from the work of Gracia-Saz and Mehta, \cite{GM08} and \cite{GM10}, where they showed that, upon non-canonical \emph{splittings}, a \vba (resp. \vbg) structure boils down to a \emph{representation up to homotopy} (\cite{AriasCrainic,AriasCrainic2}) of $A$ (resp. of $G$) on a $2$-term complex coming from the fibers of the horizontal structures.
Ordinary representations of $A$ and $G$ define particular cases of the 'up to homotopy' ones (and hence of $\mathcal{VB}$'s), but these last are strictly more general. For example, the \emph{adjoint} and \emph{coadjoint} representations of $A$, which are well known to be only 'up to homotopy' for a general $A$, arise from the intrinsic \vba structure of $TA$ and $T^*\!A$ after suitable splitting.

As an application of our main result and building on the above correspondences, we obtain a notion of \emph{integrability} of a $2$-term representation up to homotopy of a Lie algebroid $A$ and characterize the underlying obstructions by translating our main result from the world of \vbas to that of representations. (This notion of integrability coincides with the one stemming from \cite{AriasSchatz}.)

This paper is organized as follows. In section 2 we present background material on \vbgs and \vbas, including the main examples and properties. We also recall the notion of splitting of a \vba and its relation to representations up to homotopy. In section 3, we briefly review the general theory of obstructions from \cite{CF03} and then prove Theorem \ref{prop:intiff}, Proposition \ref{prop:reg type} and Theorem \ref{thm:integrals}, which are the main results of this work. In Section 4 we apply our results to obtain Corollary \ref{cor:integration reps} which provides general criteria for the integrability of two-term representations up to homotopy.  Lastly, we mention the role of integrability in the simplicial formalism for integration of representations up to homotopy given in \cite{CA01}.

\subsection*{Acknowledgements}
The authors would like to thank the PPGMA at UFPR, Curitiba, as well as IMPA, Rio de Janeiro, for supporting several visits while part of this work was carried out. Also, the authors thank Henrique Bursztyn, Matias del Hoyo, Thiago Drummond and Rajan Mehta, for interesting discussions and useful comments and suggestions that have improved this paper. Brahic was supported by CNPq (Programa Ci\^encia sem Fronteiras 401253/2012-0). Cabrera also thanks CNPq (Projeto Universal 471864/2012-9) for support. Ortiz thanks CNPq (Projeto Universal 482796/2013-8) and CAPES-COFECUB (grant 763/13) for supporting this work. The authors are very grateful to the anonymous referees for all comments and suggestions which have improved this manuscript significantly.

\section{Background material about $\mathcal{VB}$-groupoids and $\mathcal{VB}$-algebroids}\label{sec:background}

In this section we briefly recall the definition of $\mathcal{VB}$-groupoids and their infinitesimal counterparts, $\mathcal{VB}$-algebroids. There are several ways to introduce such structures, we shall follow \cite{BCdH} and do this by emphasizing the role of the underlying homogeneous structures.

Given a vector bundle $q:V \to M$, we shall denote by $m_\lambda:V \to V$ the fiberwise scalar multiplication by $\lambda \in \R$. The map $m: \R \times V \to V; (\lambda,v)\to m_{\lambda}(v)$, satisfies $m_\lambda\circ m_\mu = m_{\lambda \mu}$, $m_1 = \id$ and the regularity condition that $\frac{d}{d\lambda}m_\lambda(v)|_{\lambda=0}=0$ implies $v=m_0(v)$. Such an $\R$-action is called {\bf homogeneous structure} associated to the vector bundle structure on $V$. The base manifold $M$ can be identified with the zero section $m_0(V)$ and the bundle projection with $m_0$, namely,
\[ M \simeq m_0(V), \ q \simeq m_0:V \to M.\]
It is important to recall from \cite{GrabR} that homogeneous structures characterize the underlying vector bundle structure completely.
Moreover, a smooth map defines a vector bundle morphism iff it commutes with the underlying homogeneous structures. This point of view was used in \cite{BCdH} to give the characterizations of \vbas and \vbgs that we shall work with below.


\subsection{$\mathcal{VB}$-groupoids}
A $\mathcal{VB}$-groupoid can be roughly defined as a Lie groupoid object in the category of vector bundles. To make this more precise, consider a diagram:
\begin{align}\label{VBgroupoid}
\SelectTips{cm}{}\xymatrix{ H \ar@<0.5ex>[d]\ar@<-.5ex>[d]\ar[r]^{q_H} &G\ar@<0.5ex>[d]\ar@<-.5ex>[d]\\
           E \ar[r]^{q_E}& M}
\end{align}
where double arrows denote Lie groupoid structures and single arrows denote vector bundle structures. Denote by $m$  the homogeneous structure corresponding to the vector bundle $H\to G$ and $m^E$ the one corresponding to $E \to M$.

The original definition of a \vbg involves requiring compatibility between the Lie groupoid structure and all the structure maps underlying the vector bundle structure. (Minimal sets of axioms can be found in \cite{GM10}.) Following the perspective used in \cite{BCdH}, since the homogeneous structure $m$ encodes the vector bundle structure altogether, it turns out to be enough to ask for compatibility between the groupoid structure and $m$ only. 
\begin{definition}
A \textbf{$\mathcal{VB}$-groupoid} $(H,G,E,M)$ is a diagram like \eqref{VBgroupoid} such that $m_\lambda:H \to H$ defines a Lie groupoid morphism covering $m^E_\lambda: E \to E$ for each $\lambda \in \R$. Given two \vbgs $(H,G,E,M)$ and $(H',G',E',M')$, a {\bf morphism of \vbgs} is Lie groupoid morphism $\Phi : H \to H'$ over a map $\phi: E \to E'$ such that it commutes with the underlying homogeneous structures $(m,m^E)$ and $(m',m^{E'})$.
\end{definition}

The equivalence to other definitions can be found in \cite{BCdH}. We shall use the notation $\tilde{{\bold{s}}},\tilde{\tt},\tilde{\one}$ (resp. ${\bold{s}},\tt,\one$) for the source, target and units maps of $H$ (resp. $G$).


\begin{example}[Pair \vbg] Let $E \to M$ be any vector bundle. Then $(E\times E, M\times M, E, M)$ defines a \vbg where the underlying groupoids are the pair groupoids. The diagram \eqref{VBgroupoid} reads:
$$\SelectTips{cm}{}\xymatrix{
 E\times E  \ar@<0.5ex>[d]\ar@<-.5ex>[d]
            \ar[r]                         & M\times M \ar@<0.5ex>[d]\ar@<-.5ex>[d]\\
 E          \ar[r]                         & M.}$$
\end{example}

\begin{example}[Tangent groupoid]
Let $G\rightrightarrows M$ be a Lie groupoid. The application of the tangent functor to each of the structural maps that define $G$, gives rise to a Lie groupoid $TG\rightrightarrows TM$, referred to as the \textbf{tangent groupoid} of $G$.
 One easily checks that $(TG,G,TM,M)$ is a $\mathcal{VB}$-groupoid whose diagram is:
$$\SelectTips{cm}{}\xymatrix{
 TG  \ar@<0.5ex>[d]\ar@<-.5ex>[d]
            \ar[r]                         & G \ar@<0.5ex>[d]\ar@<-.5ex>[d]\\
 TM          \ar[r]                         & M.}$$
\end{example}

\begin{example}[Cotangent groupoid]
Let $G\rightrightarrows M$ be a Lie groupoid with Lie algebroid $A$. It was shown in \cite{CDW}, that the cotangent bundle $T^*G$ is a Lie groupoid over $A^*$. The source and target maps are defined by
$$\tilde{{\bold{s}}}(\alpha_g)u=\alpha_g(Tl_g(u-T\tt(u)))\quad \text{ and }\quad \tilde{\tt}(\beta_g)v=\beta_g(Tr_g(v))$$
where $\alpha_g \in T^*_gG$, $u\in A_{{\bold{s}}(g)}G$ and $\beta_g\in T^*_gG$, $v\in A_{\tt(g)}G$. Here, $l_g$ and $r_g$ denote the left and right multiplication by $g\in G$, respectively. The multiplication on $T^*G$ is defined by
$$(\alpha_g\circ \beta_h)(T\mu(X_g, Y_h))= \alpha_g(X_g)+ \beta_h(Y_h)$$
for $(X_g,Y_h)\in T_{(g,h)}G_{(2)}$ and $\mu:G_{(2)}=\{(g,h)\in G\times G:{\bold{s}}(g)=\tt(h)\} \to G$ denotes the multiplication map on $G$. 
We refer to $T^*G$ with the groupoid structure over $A^*$ as the \textbf{cotangent groupoid} of $G$. One observes that $(T^*G,G,A^*,M)$ is a $\mathcal{VB}$-groupoid with diagram:
$$\SelectTips{cm}{}\xymatrix{
 T^*G  \ar@<0.5ex>[d]\ar@<-.5ex>[d]
            \ar[r]                         & G \ar@<0.5ex>[d]\ar@<-.5ex>[d]\\
 A^*         \ar[r]                         & M.}$$
\end{example}

The last two examples provide a glimpse of the existence of a rich theory of duality for \vbgs and \vbas, see for example \cite{Prad2, mackenzie-book}.

\subsection{$\mathcal{VB}$-algebroids}\label{sec:vba}
As for the case of a \vbg, a \vba can be thought of as a Lie algebroid object in the category of vector bundles. More precisely, one may consider a commutative square in which vertical arrows correspond to Lie algebroid structures while horizontal arrows to the projection of a vector bundle structure:
\begin{align}\label{doubleVB}
\SelectTips{cm}{}
\xymatrix{ D \ar[d]_{p_D}\ar[r]^{p} &A\ar[d]^{p_A}\\
           E \ar[r]^{p_E}& M.}
\end{align}
We shall denote by $m$ the homogeneous structure corresponding to the top horizontal bundle $D\to A$ and by $m^E$ the one corresponding to $E\to M$. 

As for \vbgs, we follow \cite{BCdH} and require a compatibility between vertical Lie algebroid structures and the homogeneous structure $m$; the equivalence to other definitions can be also found in that reference. (See \cite{GM08} for different sets of axioms for a \vba and their equivalence.)

\begin{definition}\label{def:vba}
A \textbf{$\mathcal{VB}$-algebroid} $(D,A,E,M)$ is a diagram like \eqref{doubleVB} where the vertical arrows $D\longrightarrow E$ and $A\longrightarrow M$ are equipped with Lie algebroid structures such that $m_\lambda:D \to D$ defines a Lie algebroid morphism covering $m^E_\lambda:E \to E$ for each $\lambda \in \R$. Given two \vbas $(D,A,E,M)$ and $(D',A',E',M')$ a {\bf morphism of \vbas} is a Lie algebroid morphism $\Phi: D \to D'$ covering a map $\phi:E \to E'$ which commutes with the underlying homogeneous structures $(m,m^E)$ and $(m',m^{E'})$.
\end{definition}

Given this definition, we can introduce the notion of a {\bf double vector bundle} as a \vba $(D,A,E,M)$ encoded in diagram \eqref{doubleVB} in which the Lie algebroid structures are trivial (see also \cite{Pradines, mackenzie-book,GM08}). Hence, every \vba has an underlying \DVB{} structure, given by the same diagram \eqref{doubleVB} and by forgetting the Lie algebroid structures defined on the vertical arrows.
\bigskip

We now briefly recall general facts about structure of \DVBs and \vbas  (see \cite{mackenzie-book,GM08,GrabR} for a more detailed treatment).
Given a \DVB $(D, A, E, M)$, the vector bundles $A$ and $E$ are called
the \textbf{side bundles}. The zero sections are denoted by $0^{A}: M \longrightarrow A$, $0^{E}: M \longrightarrow E$, $0^D_{A}: A \longrightarrow D$ and $0^D_{E}: E \longrightarrow D$. 
Elements of $D$ are written $(d, a, e, m)$, where $d \in D$, $m\in M$ and $a=p(d) \in A_m$, $e=p_D(d)\in
E_m$.

In a double vector bundle, the compatibility of the homogeneous structures $m$ and $m^E$ imply (see \cite{GrabR}) that the two additions and scalar multiplications on $D$, seen as a vector bundle over either $A$ or $E$, satisfy the following \emph{interchange laws}:
\begin{align*}
(d_1 +_E d_2) +_A (d_3 +_E d_4) &= (d_1 +_A d_3)+_E (d_2 +_A d_3)\\
t \cdot_A (d_1+_E d_2)&=(t\cdot_A d_1)+_E(t\cdot_A d_2),\\
t \cdot_E (d_1+_A d_3)&=(t\cdot_E d_1)+_A(t\cdot_E d_3),\\
t \cdot_A (s \cdot_E d)  &= s\cdot_E(t \cdot_A d).
\end{align*}
 whenever these make sense, namely, for any $s,t\in \mathbb{R}$ and $d_1,d_2,d_3,d_4 \in D$ such that $(d_1,d_2) \in  D \times_E D$,  $(d_1,d_3) \in  D \times_A D,$
$(d_3,d_4) \in  D \times_E D,$  $(d_2,d_4) \in  D \times_A D$ and $d\in D$.

\begin{definition} The \textbf{core} $C$ of a \DVB is the intersection of the kernels of
$p$ and $p_D$. It has a natural vector bundle structure over
$M$, the projection of which we denote by  $p_C: C \longrightarrow M$. The
inclusion $C \hookrightarrow D$ is usually denoted by
$$
C_m \ni c \longmapsto \overline{c} \in p^{-1}(0^A_{m}) \cap p_{D}^{-1}(0^E_{m}).
$$
\end{definition}
Elements in $C_m$ can be added and multiplied by scalars using the vector bundle structure of $D$  over either  $A$ or $E$. However, as an easy consequence of the interchange laws given above, 
the two  resulting sums and multiplications do coincide. Therefore, $C$ inherits of a well defined vector bundle structure over $M$.  

This vector bundle structure can yet be understood alternatively as follows. 
Consider the short exact sequence of vector bundles over $E$:
\begin{equation}\label{eq:coreseqalg} 0 \longrightarrow K \longrightarrow D \overset{(p,p_D)}{\longrightarrow} p_E^*A\longrightarrow 0.\end{equation}
Then, the core coincides with the pull back of $K$ to $M$ via the zero section $0^{E}: M \longrightarrow E$.  Moreover, using addition in the vector bundle $D\to A$, one gets an identification $K\simeq p_E^*C$ (see eq. \eqref{core_section} below). 

\begin{example}\label{ex:TE1}
Given a vector bundle $p_E:E \to M$, the quadruple given by  $(TE,TM,E,M)$ defines a \DVB. To understand its core, recall that there is a natural identification
\[ \nu_e: E_{p_E(e)} \to \ker(T_e p_E) \subset T_e E, \ z \mapsto \frac{d}{dt}|_{t=0}( e + t z),\]
 between a fiber of the vector bundle and its vertical tangent space at a point $e\in E$. The core of $TE$ is then given by $\nu_{0^E}(E)\simeq E$, namely, the vertical tangent space at points of the zero section of $E$. Considering the standard tangent Lie algebroid structures on $TE \to E$ and $TM \to M$, $(TE,TM,E,M)$ defines a \vba structure.
\end{example}

A double vector bundle morphism $(D,A,E,M) \to (D',A',E',M')$ is, according to Definition \ref{def:vba}, a vector bundle morphism $(\Phi,\phi): (D\to E) \to (D'\to E')$ which commutes with the underlying homogeneous structures. In particular, it also defines a vector bundle morphism $(D\to A)\to (D'\to A')$ for the other side bundle structures and, hence, it induces a morphism $\Phi|_{C} : C \to C'$ between the cores of $D$ and $D'$, respectively.

An example of a \DVB morphism is the anchor map $\rho_D:D\to TE$ of a \vba $(D,A,E,M)$. When seen as a morphism $(D\to E) \to (TE \to E)$ it covers the identity on $E$ while, when seen as a morphism $(D\to A) \to (TE \to TM)$, it covers the anchor map $\rho_A: A \longrightarrow TM$. The induced morphism between the cores is \[\partial=\rho_D|_C:C\to E,\] where $E$ is identified with the vertical bundle in $TE$ at the zero section as in Example \ref{ex:TE1}. This map is called the \textbf{core anchor} of $D$\footnote{In \cite{GM08}, the authors adopt a different convention in which the core anchor is minus the map defined here.}.

%

\begin{example}
 Given a Lie algebroid $A\to M$, the tangent functor can be used to determine a $\mathcal{VB}$-algebroid structure on $(TA,A,TM,M)$ (see, e.g., \cite{mackenzie-book}). The core is identified with $A$ and the core-anchor map is $\partial=\rho_A: A \to TM$.
\end{example}

\begin{example}
Let $A\to M$ be a vector bundle. The cotangent bundle $T^*\!A$ inherits the structure of a vector bundle over $A^*$. Moreover, if $A$ is a Lie algebroid, then $T^*\!A\to A^*$ inherits a Lie algebroid structure making $(T^*\!A,A,A^*,M)$ into a $\mathcal{VB}$-algebroid.
 The Lie algebroid structure on $T^*\!A\to A^*$ can be described in Poisson-geometric terms as follows. The Lie algebroid structure on $A$ induces a linear Poisson structure on $A^*$, therefore, its cotangent bundle $T^*\!A^*\to A^*$ is equipped with a \emph{linear} Lie algebroid structure. There exists a canonical isomorphism of double vector bundles $R:T^*\!A\to T^*\!A^*$ covering the identity of $A^*$, called \emph{reversal isomorphism} in \cite{mackenzie-book} (in fact, $R$ generalizes the classical Legendre-Fenschel transform \cite{Duf}). \ale The Lie algebroid structure on $T^*\!A$ is then defined by pulling back the one on $T^*\!A^*$ via the isomorphism $R$. The underling core is identified with $T^*\!M$ and the core-anchor map with the transpose of $\rho_A$.
\end{example}

Finally, we recall the notion of \emph{core} and \emph{linear} sections on a double vector bundle. 
Given a section $c\in \Gamma C$, the corresponding \textbf{core section} $\hat{c}: E \rightarrow D$ is defined as
 \begin{equation}\label{core_section}
 \hat{c}(e_m) = 0^D_{E}(e_m) +_{A} \overline{c(m)}, \,\, m \in M, \, e_m \in E_m.
 \end{equation}
 We denote the space of core sections by $\Gamma_c(E, D)$.
 A section $\chi \in \Gamma(E,D)$ is called \textbf{linear} if $\chi: E \rightarrow D$ defines a vector bundle morphism covering a section $a:M\longrightarrow A$. The space of linear sections is denoted by  $\Gamma_{\ell}(E, D)$. The definition of a \vba implies that the following bracket conditions are satisfied:
 \begin{align}
\left[\Gamma_{\ell}(E, D), \Gamma_{\ell}(E,D)\right]_D &\subset \Gamma_{\ell}(E,D), \nonumber \\
[\Gamma_{\ell}(E,D), \Gamma_c(E,D)]_D &\subset \Gamma_c(E,D), \nonumber \\
 [\Gamma_c(E,D), \Gamma_c(E,D) ]_D &= 0 \label{eq:bracketrelations}.
\end{align}
 This can be seen directly by noticing that the Lie algebroid isomorphism defined by $(m_\lambda,m^E_\lambda)$, for $\lambda\neq 0,$ relates a linear section $\alpha$ to itself and a core section $\hat{c}$ to the core section $\widehat{(\lambda c)}$. (In fact, the above bracket relations characterize completely the \vba structure on a given \DVB, see \cite{GM08} for details.)

\subsection{Lie theory for $\mathcal{VB}$-groupoids and \vbas}\label{subsec:LieVB}
We summarize here the basics about Lie theory for \vbas and \vbgs which will allow us to state the integrability problem that is the main object of study of this paper. A detailed study of Lie theory in the $\mathcal{VB}$-setting can be found in \cite{BCdH}.

Given a Lie groupoid $G\rightrightarrows M$, we denote by $\LIE(G)\to M$ its Lie algebroid following the convention that the vector bundle underlying $\LIE(G)$ is $\ker(Ts)|_{1(M)}$ and that the Lie bracket comes from right invariant vector fields. The corresponding Lie functor going from the category of Lie groupoids to the category of Lie algebroids will be denoted by $\LIE$.
%

Not every Lie algebroid $A\to M$ is isomorphic to $\LIE(G)$ for some Lie groupoid $G$. When it is the case, we say that $A$ is \textbf{integrable} and that $G$ integrates $A$. Moreover, for an integrable Lie algebroid $A$, the ${{\bold{s}}}$-simply connected Lie groupoid integrating $A$ is unique up to isomorphism. We shall denoted by $\G(A)\rightrightarrows M$ such an ${{\bold{s}}}$-simply connected integration. (In Section \ref{sec:obstructions:integrability:algebroid}, this notation will correspond to a particular ${\bold{s}}$-simply connected integration given by the Weinstein groupoid of $A$).

Similarly, given a $\mathcal{VB}$-groupoid $(H,G,E,M)$ we obtain a $\mathcal{VB}$-algebroid structure on the quadruple $(\LIE(H),\LIE(G),E,M)$ by applying the Lie functor to the morphisms defined by $q_H:H\to G$ and $m_\lambda$ (see \cite{mackenzie-book, BCdH}). 
We obtain a functor that we also denote by:
$$\LIE:\VBG\to \VBA,$$
where $\VBG$ and $\VBA$ denote the categories of \vbgs and \vbas, respectively.

\begin{definition}\em
A $\mathcal{VB}$-algebroid $(D,A,E,M)$ is called \textbf{integrable} if there exists a $\mathcal{VB}$-groupoid $(H,G,E,M)$ such that $\LIE(H,G,E,M)$ is isomorphic to $(D,A,E,M)$ as a \vba. In this case, we say that the $\mathcal{VB}$-groupoid $(H,G,E,M)$ {\bf integrates} the $\mathcal{VB}$-algebroid $(D,A,E,M)$.
\end{definition}

A key point for the study to be developed in this paper is the following result from \cite{BCdH}: $(D,A,E,M)$ is integrable as a \vba if and only if its top Lie algebroid structure $D\to E$ is integrable as a Lie algebroid. More precisely,

\begin{theorem}[\cite{BCdH}] \label{thm:DintegGinteg}
Let $(D,A,E,M)$ be a $\mathcal{VB}$-algebroid. If the Lie algebroid $D\to E $ is integrable, then the Lie algebroid $A\to M$ is integrable and $(\G(D),\G(A),E,M)$ admits a natural $\mathcal{VB}$-groupoid structure integrating $(D,A,E,M)$ as a \vba.
\end{theorem}
We mention that the construction of the \vbg structure involves lifting the $\R$-action $m$ underlying $D\to A$ to $\G(D)$ by means of Lie's second theorem and showing that it fulfills the regularity condition of an homogeneous structure. In the context of the present paper, Theorem \ref{thm:DintegGinteg} will allow us to study the obstructions for the integrability of a \vba by focusing on the Lie algebroid $D\to E$.

To end this subsection, we describe two natural operations in the category of vector bundles that restrict in a straightforward manner to \vbgs and \vbas, namely, \emph{pull-backs} and (in particular) \emph{direct sums}.

 Let $(D,A,E,M)$ be a \vba and $(H,G,E,M)$ be a \vbg.
Given another \vba $(D', A,E',M)$, the sum $(D\oplus_A D', A, E\oplus_M E', M)$ inherits a natural \vba structure. Analogously, using the obvious notations, $(H\oplus_G H', G, E\oplus_M E', M)$ inherits a natural \vbg structure and  we have:
\begin{multline}
\LIE\left(H\oplus_G H', G, E\oplus_M E', M\right) \\= \left(\LIE(H)\oplus_{\LIE(G)} \LIE(H'), \LIE(G), E\oplus_M E', M\right).	
\end{multline}

More generally, given a Lie algebroid morphism $\phi:B \to A$ covering $f:N\to M$ and denoting $\phi^*D \to B$ the pull-back bundle associated to the vector bundle $D\to A$, one can show that $(\phi^*D,B,f^*E,N)$ inherits a natural \vba structure. In the groupoids category, if $F: G' \to G$ is a Lie groupoid morphism covering $f:N\to M$, one shows that $(F^*H,G',f^*E,N)$ inherits a natural \vbg structure and that
$$
\LIE\left(F^*H,G',f^*E,N\right) = \left(\LIE(F)^*\LIE(H), \LIE(G'),f^*E,N\right).
$$
 A systematic treatment of pull-backs in the categories of Lie groupoids and algebroids, together with their behavior with respect to $\LIE$, can be found in \cite{BCdH}.

\subsection{Split \vbas and representations up to homotopy}\label{subsec:split}

Here we follow closely the reference \cite{GM08} to recall some facts about splittings of \vbas and their relation to representations up to homotopy. 

First, notice that given a \DVB $(D,A,E,M)$, there are three vector bundles over $M$ associated to it: the two sides $A$ and $E$ and the core $C$. Conversely, given three vector bundles $A,E,C$ over $M$, their fiber product
\[ \Dd = A\times_M E \times_M C\]
admits a double vector bundle structure with sides $A,E$ and core $C$ given by the natural set-theoretic identifications \[\Dd \simeq p_A^*E\oplus_A p_A^*C \simeq p_E^*A \oplus_E p_E^*C.\] A \DVB of this form is said to be \emph{split} and we shall use the symbol $\Dd$ to denote it.

\begin{definition}\label{def:splitting}
Given a \DVB $(D,A,E,M)$ with core $C$, a {\bf splitting} of $(D,A,E,M)$ is a \DVB isomorphism $h: \Dd \overset{{}_\sim}{\longrightarrow} D$ inducing the identity morphisms on $A$, $E$ and $C$. 
\end{definition}
Splittings can be shown to exist for any double vector bundle; let us review the argument given in \cite{GM08}. First, we observe that the above splittings  correspond to a special subset of splittings of the sequence \eqref{eq:coreseqalg} whose dependence on the $E$-fibers is also \emph{linear}. 
Grabowski and Rotkiewicz \cite{GrabR} proved that these exist locally (in $M$) and \cite{GM08} proposed to infer the global existence  via a \v{C}ech cohomological argument.

Splittings also form an affine space modeled on $\Hom(A\otimes E, C)$ which, in turn, can be seen as the  group of  splittings of $\Dd$ itself acting by
\begin{equation}\label{eq:changespl0} A\times_M E \times_M C \ni (a,e,c) \mapsto (a,e,c + \sigma(a,e)), \ \sigma \in \Hom(A\otimes E, C).\end{equation}
Thus, any two splittings of a given \DVB differ by a transformation of $\sigma \in \Hom(A\otimes E,C)$ as above.

\begin{example}\label{ex:splitTE}
 Let $p_E:E \to M$ be a vector bundle and $(TE,TM,E,M)$ be the associated \DVB. Splittings of $TE$ are in one-to-one correspondence with linear $TM$-connections on $E$. Indeed, a $TM$-connection $\nabla$ on $E$ defines a horizontal lift for each $e \in E$,
\[hor^\nabla_e: T_{p_E(e)} M \to T_e E, X \mapsto \frac{d}{dt}|_{t=0} \hol^\nabla_{x(t)} e,\]
 where $x(t)$ is a curve in $M$ with velocity $X$ at $t=0$ and $\hol^\nabla_{x(t)}: E_{x(0)}\to E_{x(t)}$ is the linear holonomy defined by $\nabla$.   The tangent bundle of $E$ is thus split into horizontal and vertical parts and the induced map
\[h: TM\times_M E \times_M E \to TE, (X,e,z) \mapsto hor^\nabla_e(X) + \frac{d}{dt}|_{t=0}(e + tz),\]
defines a bijection. 
The fact that $hor^\nabla$ is linear implies that $h$ defines a \DVB morphism (i.e. it is locally given by a transformation of the form \eqref{eq:changespl0}). It evidently induces the identity on side bundles and, to check that it is also the identity on the cores, one observes that $e=0^E$ and $X=0^{TM}$ implies $hor^\nabla_e(X)=0_E^{TE}$ and recalls the identification of $E$ as the core of $TE$ (c.f. Example \ref{ex:TE1}). Conversely, given $h$ one can define an underlying $TM$-connection $\nabla$ by defining $h(X,e,0^E)$ to be the horizontal lift $hor^\nabla_e(X)$ of $X\in TM$ and the fact that $h$ preserves the \DVB structure implies that the induced holonomy (and, hence $\nabla$) is linear.
\end{example}

Given a \vba $(D,A,E,M)$ and a splitting $h$, there is an induced \vba structure $(\rho_{D,h},[\ ,\ ]_{D,h})$ on $\Dd$ which we shall call \emph{split}. We now describe such split \vba structures and relate them to representations up to homotopy.

The first observation is that the induced core-anchor $\partial:C \to E$ on $\Dd$ is the same as the one for $D$, since $h$ preserves cores. To characterize the rest of the \vba structure, we observe that there are natural maps of sections
\begin{align*}
\Gamma(A) \to&\  \Gamma(E,\Dd), \ \alpha \mapsto \alpha^l, \ \alpha^l(e) = \left(\alpha \circ p_E(e), e , 0^C\right) \\
\Gamma (C)  
\to &\  \Gamma(E, \Dd), \ c \mapsto \hat{c}, \ \hat{c}(e) = \left(0^A, e , c\circ p_E(e)\right),
\end{align*}
defining linear and core sections of $\Dd$, respectively.
Following \cite{GM08} further, the formulas 
\begin{eqnarray}
 \widehat{\nabla^C_\alpha c} &= &[\alpha^l, \hat{c}]_{D,h}, \nonumber \\
 \langle \xi, \nabla_\alpha^E e\rangle &=& \rho_A(\alpha)\langle \xi, e \rangle - \langle \mathcal{L}_{\rho_{D,h}(\alpha^l)} \xi, e\rangle, \nonumber \\
 \widehat{\omega(\alpha,\beta)(e)}&=& [\alpha^l,\beta^l]_{D,h}(e)-[\alpha,\beta]_A^l(e) , \label{curvaturehorlift}
\end{eqnarray}
define $A$-connections $\nabla^C$ and $\nabla^E$ on $C$ and $E$, respectively, and a curvature-like object $\omega \in \Gamma(\wedge^2A^*)\otimes \mathrm{\Hom}(E,C)$, where $\alpha \in \Gamma (A)$, $e \in \Gamma (E)$, $c\in \Gamma (C)$ and $\xi \in \Gamma (E^*)$ is also seen as a fiberwise-linear function on $E$. Notice that the formulas above make sense because of the bracket relations \eqref{eq:bracketrelations}.

We will refer to $(\nabla^E,\nabla^C,\omega)$ as the {\bf connection data} defined by the splitting $h$ of $(D,A,E,M)$. Recall that an {\bf $A$-connection}  on a vector bundle $V\to M$ is an $\R$-bilinear map $\nabla: \Gamma(A) \times \Gamma(V) \to \Gamma(V)$ satisfying $\nabla_{f\alpha} s = f \nabla_\alpha s$ and $\nabla_\alpha(fs) =  f \nabla_\alpha s + \mathcal{L}_{\rho_A(\alpha)}(f) s$, for all $\alpha \in \Gamma(A), s\in \Gamma(V),$ and $f\in C^\infty(M)$. An $A$-connection defines an operator $d_\nabla$ on $V$-valued $A$-forms $\Omega^\bullet(A;V)=\Gamma(\wedge^\bullet A^* \otimes  V)$ by the formula
\begin{eqnarray}
 d_\nabla \omega (\alpha_1, \dots,\alpha_{k+1})&=&\sum_{i<j} (-1)^{i+j} \omega([\alpha_i,\alpha_j], \dots, \hat{\alpha}_i,\dots, \hat{\alpha}_j,\dots,\alpha_{k+1}) + \nonumber\\
 & & + \sum_i (-1)^{i+1} \nabla_{\alpha_i} \omega(\alpha_1,\dots,\hat{\alpha}_i,\dots,\alpha_{k+1}).\label{eq:dnabla}
 \end{eqnarray}
This operator is of degree $1$ and a derivation for the natural $\Omega(A)$-module structure,
\[d_\nabla(\alpha s)=(d_A\alpha)s +(-1)^p \alpha (d_\nabla s),\]
where $d_A$ denotes the Chevalley-Eilenberg differential on $\Omega^\bullet(A)$, $\alpha \in \Omega^p(A)$ and $s\in \Gamma (V)$. Moreover, the curvature of $\nabla$ is $R_\nabla\equiv d_\nabla^2$ so that the $A$-connection is \emph{flat} iff $d_\nabla^2=0$.

It turns out that the \vba structure on $\Dd$ is completely characterized by the connection data together with the core-anchor $\partial : C \to E$.

\begin{theorem}{(\cite{GM08})}\label{lma:conn data} Let $A\to M$ be a Lie algebroid and $E, \ C$ vector bundles over $M$. The  formulas \eqref{curvaturehorlift} define a $1:1$ correspondence between split \vba structures $(\Dd,A,E,M)$ with core anchor $\partial:C \to E$ and quadruples $(\partial, \nabla^E,\nabla^C,\omega)$ as above satisfying the following relations: 
\begin{align*}
\partial \circ \nabla^C &= \nabla^E \circ \partial,\\
-\omega\circ \partial &= R_{\nabla^C}, \\
-\partial \circ \omega &= R_{\nabla^E}, \\
d_{\nabla^{\scriptscriptstyle \Hom(E, C)}}\omega &=0,
\end{align*}
  where $R_{\nabla^{E}}$ and $R_{\nabla^{C}}$ denote the curvatures of $\nabla^{E}$ and $\nabla^{C}$ respectively\footnote{Notice additional minus signs with respect to the formulas in \cite{GM08} coming from different sign conventions in the definitions 
 \eqref{curvaturehorlift} and of $\partial$. The consistency of the second and third equations in Thm. \ref{lma:conn data} with our present definitions can be directly verified computing the Jacobi identity on linear-linear-core sections and the identity saying that $\rho_{D,h}$ preserves brackets for two linear sections, respectively. We thank the referee for this observation.  },  while $\nabla^{\scriptscriptstyle \Hom(E, C)}$ denotes the $A$-connection on $\Hom(E,C)$ naturally induced by $\nabla^E$ and $\nabla^C$.
\end{theorem}

Let us move on to the interpretation of the split \vba structure in representation-theoretic terms. Recall that a \emph{representation} of $A$ on a vector bundle $V$ is a \emph{flat} $A$-connection $\nabla$ on $V$. From Theorem  \ref{lma:conn data}, we see that the connection data will not define representations of $A$ on $E$ and $C$ in general. Nevertheless, it can be understood as a \emph{representation up to homotopy} of $A$ on the $2$-term complex \[C \overset{\partial}{\to} E,\]
where the curvature of the connections $(\nabla^E,\nabla^C)$ is controlled by $\omega$ and $\partial$  (see the Section \ref{sub:comp:obs} for an interpretation in terms of holonomies and chain homotopies.)

\begin{remark}[Induced representation in cohomology]\label{rmk:repincohom}
 According to a general principle, a representation \emph{up to homotopy} on a chain complex should induce an honest representation in cohomology. Indeed, notice that the cohomology of $\partial:C \to E$ seen as a $2$-term complex is $\ker \partial \oplus \coker \partial$. Whenever these define smooth bundles, the relations in Theorem \ref{lma:conn data} imply that the connections $\nabla^C$ and $\nabla^E$ induce $A$-connections $\nabla^{\ker \partial}$ and $\nabla^{\coker \partial}$ on  $\ker \partial$ and $\coker \partial$, respectively, and that these are \emph{flat}. 
\end{remark}

To introduce a formal definition of a representation up to homotopy, consider  $\E = C\oplus E$ the graded vector bundle with $C$ in degree $-1$ and $E$ in degree $0$. In analogy with an ordinary connection, an {\bf $A$-superconnection} on $\E$ is an operator $\D: \Omega^\bullet(A) \otimes \Gamma(\E) \to \Omega^\bullet(A) \otimes \Gamma(\E)$ such that it is of degree $1$ with respect to the total grading and that it is a graded derivation of the natural $\Omega^\bullet(A)$-module structure. The superconnection is said to be {\bf flat} if $\D^2=0$. 
A (2-term) {\bf representation up to homotopy} of $A$ on $\E$ is a flat $A$-superconnection on $\E$ (see \cite{AriasCrainic,GM08}).

Such a  differential $\D$ can be decomposed as
\[ \D = \D_0 + \D_1 + \D_2,\]
where, setting $\E=\oplus_l \E_l$, $\D_i$ sends $\Omega^k(A)\otimes \Gamma(\E_l)$ to $\Omega^{k+i}(A) \otimes \Gamma(\E_{l-i+1})$. Following \cite{GM08}, $\D_0$ is related to the core-anchor $\partial$, $\D_1$ to $d_{\nabla^E}$ and $d_{\nabla^C}$, and $\D_2$ to multiplication by $\omega$ as follows: 
\begin{equation}\label{eq:Drep} \D(c + e) = \partial c + d_{\nabla^C} c + d_{\nabla^E} e + \omega(e), \ \ c\in \Gamma (C), e \in \Gamma (E).\end{equation}
The relations given in Theorem \ref{lma:conn data} are equivalent to the flatness condition $\D^2=0$. 

In this way, $2$-term representations up to homotopy are in $1:1$ correspondence with split \vba structures $(\partial,\nabla^E,\nabla^C,\omega)$ on $\Dd$. We shall come back to representations up to homotopy in Section \ref{sec:reputh}.

\begin{remark}[Change of splitting]\label{rmk:changesplit}
 Recall that changes of splitting of a given \DVB correspond to elements $\sigma \in \Hom(A\otimes E,C)$ acting as in the equation \eqref{eq:changespl0}. Given a \vba structure on $\Dd$ characterized by $(\partial,\nabla^E,\nabla^C,\omega)$, a change of splitting $\sigma$ transforms the data into $(\tilde{\partial},\widetilde{\nabla}^E, \widetilde{\nabla}^C, \tilde{\omega})$ where:
 \begin{align}
  \tilde{\partial}=\partial&, \ \ \ \widetilde{\nabla}^C_\alpha - \nabla^C_\alpha =- \sigma_\alpha \circ \partial, \ \ \ \widetilde{\nabla}^E_\alpha - \nabla^E_\alpha =- \partial\circ \sigma_\alpha \nonumber  \\
  \tilde{\omega}_{\alpha,\beta}&= \omega_{\alpha,\beta} - \sigma_{[\alpha,\beta]_A} + \sigma_\alpha \nabla^E_\beta - \sigma_\beta \nabla^E_\alpha +
  \nabla^C_\alpha \sigma_\beta - \nabla^C_\beta \sigma_\alpha - \sigma_\alpha \partial \sigma_\beta + \sigma_\beta \partial \sigma_\alpha, \label{eq:changespl1}
 \end{align}
where $\alpha,\beta \in \Gamma (A)$ and $\sigma_\alpha, \sigma_\beta$ are seen as sections of $\Hom(E,C)$. Details can be found in \cite[Thm. 4.14]{GM08}.
\end{remark}

\begin{remark}[The Chevalley-Eilenberg differential of $D$]\label{rmk:CEdiff}
 Every Lie algebroid $A$ has a differential induced on its space of forms $\Gamma (\wedge^\bullet A^*)$. In the case of a \vba $(D,A,E,M)$, we denote by $D^*_E$ the dual with respect to the vector bundle $D\to E$ and then we get a differential $d_D$ on $\Gamma(E, \wedge^\bullet D_E^*)$. Suppose that $D=\Dd$ is split. Then, $D^*_E = A^*\times_M E\times_M C^*$ and we have the following identification:
 \[\Gamma(E, \wedge^\bullet D_E^*) = \Gamma (\wedge^\bullet A^*) \otimes C^\infty(E) \otimes \Gamma (\wedge^\bullet C^*).\]
 Following \cite{GM08}, the subspace $\Gamma (\wedge^\bullet A^*) \otimes \Gamma (E^*) \otimes \Gamma (C^*)\subset \Gamma(E, \wedge^\bullet D_E^*)$ made of linear functions on $E$ and degree $1$ forms on $C$ is invariant under $d_D$. This subcomplex is in duality with $\Omega^\bullet(A) \otimes \Gamma(\E)$ as $\Omega^\bullet(A)$-modules (notice that $\E^*=E^*\oplus C^*$) and, moreover, $d_D$ is the dual of the superconnection $\D$:
 \[ \langle \D \omega, \eta\rangle = d_A\langle \omega,\eta\rangle -(-1)^{|\omega|} \langle \omega, d_D \eta\rangle,\]
 for $\omega \in \Omega^\bullet(A)\otimes \Gamma(\E)$ and $\eta \in \Omega^\bullet(A)\otimes \Gamma(\E^*)$.
\end{remark}


\section{Integrability of \vbas}\label{sec:integrability}

In this section, we explain how the general theory of integrability for Lie algebroids of \cite{CF03} specializes to the case of \vbas. We also make use of the structure of regular \vbas introduced in \cite{GM08} to provide integral expressions for the obstructions to the integrability.

\subsection{General theory of obstructions}\label{sec:obstructions:integrability:algebroid}

Here, we briefly recall the construction of the Weinstein groupoid of a Lie algebroid and of the monodromy groups that control its smoothness. The reader is referred to \cite{CF03} for a full exposition.

Let $A$ be a Lie algebroid over $M$, with anchor $\rho_A:A\to TM$. We consider the interval $I=[0,1]$ and denote by $t$ the standard parametrization of $I$. Any vector bundle map $TI\to A$ can be written as $a dt$, with $a: I\to A$ a path covering $\gamma:=p_A\circ a:I\to M$ on the base manifold. Such a map  $adt:TI \to A$ defines a Lie algebroid morphism iff \[\rho_A\circ a=\frac{d}{dt}\gamma,\] in which case $a$ is called an {\bf $A$-path}. We will denote by $P(A)$ the set of all $A$-paths.

Let us now consider a smooth family of $A$-paths $a^{s}:I\to A$, parametrized by $s\in I$, such that the base end-points $\gamma^s(0)$ and $\gamma^s(1)$ are fixed.  In that case, there exists \cite{CF03} a unique family of $A$-paths $b^tds:TI \to A$, with $b^{0}(s)=0$ and such that $adt+bds:TI^2\to A$ is a Lie algebroid morphism. (Here $I^2:=I \times I$ is parametrized by $(t,s)$.)
The family $a^{s}$ is called an {\bf $A$-homotopy} between the $A$-paths $a^{0} dt$ and $a^{1} dt$ if the homotopy condition $b^{1}(s)=0, \forall s\in I$ is satisfied. 
One may also define an $A$-homotopy directly as a Lie algebroid morphism \[adt+bds:TI^2\to A\] satisfying the boundary conditions: $b|_{t=0,1}=0$.

%

Given an arbitrary (not necessarily integrable) Lie algebroid, one can always consider the Weinstein (or fundamental) groupoid associated to $A$, which is defined as:
$$
\G(A) := P(A)\, /\! \sim
$$
where $\sim$ denotes the equivalence relation defined by $A$-homotopies. The set $P(A)$ can be endowed with the structure of a Banach manifold from which $\G(A)$ inherits a topology. The concatenation of $A$-paths endows $\G(A)$ with the further structure of a topological groupoid over $M$, which is source simply connected by construction. The source and target maps correspond to the initial and final points, respectively, of the base map $\gamma:I \to M$ of the $A$-path.

The Weinstein groupoid is universal in the sense that a Lie algebroid admits a smooth integration $G$ (i.e. $\LIE(G)\simeq A$) if and only if $\G(A)$ is a Lie (hence smooth) groupoid. In this case, there is a covering map $\G(A)\to  G$, and $\LIE(\G(A))\simeq A$ as well.


As explained in \cite{CF03}, integrability of a Lie algebroid $A\to M$ is controlled by its \textbf{monodromy groups}. Given a point $x \in M$, we denote by $\g^A_x =\ker(\rho_A|_x)$ the isotropy Lie algebra at $x\in M$. The monodromy group at $x$ consists of all elements $a\in \ker(\rho_A|_x)$ which, considered as constant $A$-paths, are $A$-homotopic to a trivial path, that is: 
 $$\Mon(A)_x:=\left\{a\in \g^A_x : \ a \sim 0_x \right\}\subset  \g^A_x ,$$  
where both $a,\,0_x \in P(A)$ denote constant $A$-paths. The union of the monodromy groups over $x\in M$ is usually denoted by $\Mon(A)\subset A$. The relevance of the monodromy groups to the integrability problem is given by the theorem below.

\begin{theorem}[\cite{CF03}] Let $p_A:A\to M$ be a Lie algebroid and $\G(A)$ its Weinstein groupoid. The following assertions are equivalent:
\begin{enumerate}[i)]
\item $A$ is integrable.
\item $\G(A)$ is smooth.
\item Any sequence $(a_n) \subset \Mon(A)$ converging to a trivial element $0_x$ consists of trivial elements for $n$ big enough.
\end{enumerate}
\end{theorem}

\subsection{Obstructions to the integration of \vbas}
We now turn to the problem of integrability of a \vba: given a \vba $(D,A,E,M)$ we address the problem of the existence of a $\mathcal{VB}$-groupoid $(H,G,E,M)$ integrating it.

As a first approximation, we remind the reader that $(D,A,E,M)$ is integrable if $D$ is integrable as a mere Lie algebroid over $E$. Indeed, if $D$ is integrable, then Theorem \ref{thm:DintegGinteg} guarantees that the $s$-simply connected Lie groupoid integrating $D$ comes canonically equipped with the structure of  a $\mathcal{VB}$-groupoid integrating $(D,A,E,M)$.
 

 The next observation is that for $D$ to be integrable, $A$ needs to be integrable as well.
 
\begin{proposition}\label{prop:Aint} Let $(D,A,E,M)$ be a \vba. If $D$ is integrable then $A$ is integrable.
\end{proposition}
\begin{proof}
It follows from the definition of a \vba that the zero section $0^D_A:A\hookrightarrow D$ defines a Lie algebroid morphism. This makes $A$ into a Lie subalgebroid of $D$. In particular, $A$ is integrable whenever $D$ is integrable. 
\end{proof}

It is important to observe that the converse of Proposition \ref{prop:Aint} does not hold in general (see Example \ref{ex:non:int} below). In the sequel, we shall always assume that $A$ is integrable, and then we find obstructions for $D$ being integrable. We start with a general observation regarding the integrability of Lie algebroids.

\begin{lemma}\label{lma: gen int}
Let $A_1\to M_1$ and $A_2\to M_2$ be Lie algebroids, and $\phi: A_1 \to A_2$ a Lie algebroid morphism. If $A_2$ is integrable and $\
\Mon(A_1)_{x_1} \cap \ker(\phi|_{x_1}) =\{ 0^{A_1}_{x_1} \}$ for all $x_1 \in M_1$, then $A_1$ is integrable.
\end{lemma}
\begin{proof}
Consider a sequence of monodromy elements $(v_n)_{n\in \mathbb{N}}\subset \Mon(A_1)$  converging to $0^{A_1}_x$ for some $x\in M_1$. Then $(\phi(v_n))_{n\in \mathbb{N}}$ is a sequence of elements of $\Mon(A_2)$ converging to $0_{\phi(x)}^{A_2}$. Since $A_2$ is integrable, $\phi(v_n)=0_{x_n}^{A_2}$ for $n$ big enough, thus $v_n\in \ker \phi=\{0_{x_n}^{A_1}\}$ for $n$ large, which proves that $A_1$ is integrable.
\end{proof}

 From a general perspective, Lemma \ref{lma: gen int} may seem far from an optimal criteria for integrability (think of the identity map $A \to A$, or a trivial map $A\to 0$). However, in the case of a \vba, it turns out to give a sufficient condition.

\begin{theorem}\label{prop:intiff} Let $(D,A,E,M)$ be a \vba. Then, $D$ is integrable if and only if A is integrable and $\Mon(D)\cap \ker p=\{0^D_E\}$.
\end{theorem}
\begin{proof}
Assume that $A$ is integrable and $\Mon(D)\cap \ker p=\{0^D_E\}$. Since $p:D\to A$ is a Lie algebroid morphism, we conclude that $D$ is integrable by Lemma \ref{lma: gen int}. Conversely, if $D$ is integrable, then Proposition \ref{prop:Aint} implies that $A$ is integrable as well. Now, suppose that there exists a non trivial element $\xi\in\Mon(D)_e\cap \ker p|_e$, where $e\in E$. 
Consider  a sequence $(\lambda_n)_{n\in \mathbb{N}}\subset \mathbb{R}$ of non vanishing numbers converging to zero in $\mathbb{R}$. Then $(m_{\lambda_n}(\xi))_{n\in\mathbb{N}}$ is a sequence of non trivial elements of $\Mon(D)$ which converges to $0^A_x\in A\subset D$, where $x=p_E(e)$. This contradicts the integrability of $D$. 
\end{proof}

Below, we list a few immediate consequences of Theorem \ref{lma: gen int}. To that end, we recall that a Lie algebroid $A$ can be restricted to one of its leafs $L\subset M$ and that, denoting by $A_L\equiv A|_L \to L$ the resulting Lie algebroid, the inclusion $\SelectTips{cm}{}\xymatrix@C=12pt{i_{A_L}:A_L\ \ar@{^{(}->}@<-1.5pt>[r] & A}$ defines a Lie subalgebroid. Considering the Lie algebroid morphism $p:(D\to E) \to (A\to M)$, it follows that the preimage $D|_{A_L}=p^{-1}(A_L)$ of the subalgebroid $A_L\subset A$ defines a Lie subalgebroid of $D$. In particular, $D|_{A_L}$ inherits a Lie algebroid structure over $E|_L$ and, by considering the restriction of the homogeneous structure of $D$, it follows that $(D|_{A_L},A_L,E|_L,L)$ defines a \vba. This structure can be alternatively understood as a pull-back of $(D,A,E,M)$ along the map $i_{A_L}$,
\begin{equation}\label{eq:restriction}(D|_{A_L}:=i_{A_L}^*D,A_L,E|_L,L).\end{equation}

%
%

\begin{corollary}\label{cor: leaves}
Let $D$ be a \vba over $A$. Then $D$ is integrable if and only if $A$ is integrable and $D|_{A_L}\to E|_L$ is integrable for each leaf $L\subset M$ of $A$.
\end{corollary}
\begin{proof}
The image  of a Lie algebroid morphism $TI^2\to A$ lies entirely over a single leaf of $A$ since $I^2$ is connected. This morphism then co-restricts to a map $TI^2\to A_L$ for some leaf $L$ of $A$. In particular, this holds for $D$-homotopies as well and, therefore, the monodromy groups of $D$ over points of $E|_L$ coincide with the monodromy groups of $D|_{A_L}$. Then, the condition in Theorem \ref{prop:intiff} needs to be checked only on $D|_{A_L}$ for each leaf $L$ of $A$.
\end{proof}

\begin{remark}
Corollary \ref{cor: leaves} may seem somewhat unexpected. Indeed, the obstructions to the integrability of an arbitrary Lie algebroid have both a longitudinal and a transverse nature with respect to the base foliation. In practice, this means that one has to compute the monodromy groups leafwise, and then make sure that they do not admit any accumulation point when moving from one leaf to another. From this point of view, a \vba shows a much more rigid behavior, since one has to check the \emph{vanishing} of the intersection $\Mon(D)\cap \ker p$, which can be done independently over $A_L$ for each leaf $L$.
\end{remark}

\begin{corollary}\label{int:injectivevba}
Let $D$ be a \vba over $A$ with core $C$ and side $E$. If the core anchor $\partial:C\to E$ is injective, then $D$ is integrable iff $A$ is.
\end{corollary}
\begin{proof}
Recall that by construction, the monodromy group $\Mon(D)_e$ of a Lie algebroid $D$ at some point $e\in E$ lies in the isotropy algebra $\g^D_e=\ker (\rho_D|_e)$ at $e$. Let us now look closer into $\ker p$. First, notice that $\rho_D$ sends $\ker p$ to $\ker Tp_E$ and that, as for any arbitrary \vba, the observation given below the  sequence \eqref{eq:coreseqalg} shows that   \[\ker p\simeq p_E^*C\overset{set}{\simeq}  E\oplus C, \ \ \ker Tp_E\simeq p_E^*E\overset{set}{\simeq}  E\oplus E.\] Then, the restriction of $\rho_D$ to $\ker p$ takes the form $\rho_D(e,c)=(e,\partial c)$. It follows that $\Mon(D)_e\cap \ker p|_e\subset \{e\}\times \ker \partial$ for any $e \in E$. When the core anchor is injective, the condition $\Mon(D)\cap \ker p=\{0^D_E\}$ is therefore automatically satisfied.
\end{proof}

The following example illustrates the computation of monodromy groups in a Poisson geometric context and, in particular, how to apply Theorem \ref{prop:intiff}. 

\begin{example}[An integrable \vba]\label{ex:int}
Consider the dual vector space $\su$ of the Lie algebra $\mathfrak{su}_2$ and let $E=C=\mathfrak{su}_2^*\times \mathbb{R}$ be seen as trivial line bundles over $M=\su$. The linear Poisson structure $\pi_M\in \mathcal{X}^2(\su)$ given by
\[ \pi_M = x^1 \partial_{x^2}\wedge \partial_{x^3} + x^2 \partial_{x^3}\wedge \partial_{x^1}+x^3 \partial_{x^1}\wedge \partial_{x^2},  \]
where $x^i$ denote standard linear coordinates on $\su\simeq \R^3$, defines a (integrable) Lie algebroid $A = T^*_{\pi_M} \su$ over $M$. On the split \DVB $D=D_{A,E,C}$ given by 
$$\SelectTips{cm}{}
\xymatrix@R=20pt@C=25pt{ *+[c]{D=T^*\su \times \R \times \R}\ar[d]_-{}\ar[r]^-{ p} &*+[c]{T^*\mathfrak{su}_2^*}\ar[d]^{ }\\
           *+[c]{\mathfrak{su}_2^*\times \mathbb{R}}\ar[r]^{ }                          &*+[c]{\mathfrak{su}_2^*}}$$
we introduce a \vba structure $(D,T^*_{\pi_M} \su,E,\su)$ by setting
\begin{align*}
&&&&   [dx^i,dx^j]_D&=dx^k+x^k \ \be \ \cc,& \rho_D(dx^i)&=x^k\partial_{x^j}-x^j\partial_{x^k},\\
&&&&   [dx^i,\cc]_D&=0,        & \rho_D(\cc)&=0,                     &&
\end{align*}
where $(i,j,k)$ denotes a cyclic permutation of $(1,2,3)$, $\be:E \to \R$ denotes the fiberwise linear coordinate, $c\equiv 1 \in \Gamma C$ is the constant section with $\hat{c}\in \Gamma(E,D)$ the corresponding core section  and $dx^i\equiv(dx^i)^l$ are sections of $T^*\!M$ seen as linear sections of $D$ (recall Section \ref{subsec:split}). (Notice that the core anchor $\partial$ of this \vba is zero.)

Let us verify that $\Mon(D)_e\cap \ker(p)=0^D_E(e)$ and thus, by Theorem \ref{prop:intiff}, that $D$ is integrable. To that end, we shall introduce a Poisson structure $\pi_E$ on the total space of $E$ and define a Lie algebroid isomorphism $\psi:(T_{\pi_E}^*E \to E) \longrightarrow (D\to E)$. We do this in order to reduce the computation of the monodromy groups $\Mon(D) \simeq^{\psi} \Mon(T_{\pi_E}^*E)$ to that of the variation of symplectic areas along the leaves of the Poisson manifold $(E,\pi_E)$, as explained in \cite{CFpoisson}. 

The Poisson structure on $E$ is given by introducing a re-scaling factor depending on the fiber coordinate
\[ \pi_E = \left(1+\be^2\!/2\right) \pi_M \in \mathcal{X}^2(E=M \times \R).\]
The corresponding cotangent Lie algebroid  $T_{\pi_E}^*E \to E$ is given by
\begin{align*}
[dx^i,dx^j]_{\pi_E}&=\left(1+\be^2\!/2\right)dx^k +  x^k \ \be d\be,\\ 
[dx^i,d\be]_{\pi_E}&=0,  \\     
\rho_{\pi_E}(dx^i)&=\left(1+\be^2\!/2\right)(x^k\partial_{x^j}-x^j\partial_{x^k}),\\
 \rho_{\pi_E}(d\be)&=0,                     
\end{align*}
where $(i,j,k)$ are again cyclic permutations of $(1,2,3)$.
It is easy to see that, defining $\psi:(T^*\!E\to E) \to (D\to E)$ as the vector bundle morphism over the identity which satisfies
\[ \psi(dx^i) = \left(1+\be^2\!/2\right) dx^i, \ \ \ \psi(d\be) = \left(1+\be^2\!/2\right)^2 \cc,\]
then $\psi$ defines a Lie algebroid isomorphism. (Notice, though, that $\psi$ is \emph{not} a \DVB morphism for the cotangent \DVB structure on $T^*E$.)
The symplectic foliation of $\pi_E$ is given by singular leaves  which are reduced to a point of the form $\{(0_{\mathfrak{su}_2^*},\be)\}$, and by symplectic spheres of the form $S^2_r\times\{\be\}$ with $S^2_r\subset \su\simeq \R^3$ the sphere of radius $r=||(x^1,x^2,x^3)||$ having symplectic area given by
$$\mathcal{A}(r,\be)=4 \pi r(1+\be^2\!/2)^{-1}.$$
 By \cite[Prop. 5]{CFpoisson}, the monodromy groups of $T_{\pi_E}^*\!E$ are given by differentiating $\mathcal{A}$ in the transverse directions:
$$
\Mon(T^*\!E)=
 4\pi\left(\frac{dr}{1+\be^2\!/2}-\frac{r\be d\be}{(1+\be^2\!/2)^{2}}\right)\cdot \mathbb{Z},\quad\text{ if } r\neq 0.\\
$$
(On the subset of singular leaves $\{r=0\}$, $\Mon(T^*\!E)$ clearly vanishes.) By applying the isomorphism $\psi$ we deduce that the monodromy groups of $D$ are given by:
$$
\Mon(D)=\begin{cases}
\ 4\pi \left(r^{-1} \sum_i x^idx^i -r \be \cc\right)\cdot \mathbb{Z},&\text{ on } \{r\neq 0\},\\
    \{0^D_E\}&\text{ on } \{r=0\}.\end{cases}$$
In conclusion, elements in $\Mon(D)\cap \ker p$ must have $x^i=0$, leaving $0^D_E(e)$ as the only possibility and thus showing that $D$ is integrable.
\end{example}

\subsection{\vbas of regular type}\label{subs:regular}
In this Section we shall restrict our study of integrability to the sub-class of \vbas called \emph{regular}. It will turn out that this class already encodes all the ingredients needed for the general case.

First, let us recall some basic definitions from \cite{GM08}.
\begin{definition}\label{def:type01}
A $\mathcal{VB}$-algebroid $(D,A,E,M)$  is called \textbf{regular} if the core anchor $\partial: C \to E$ has constant rank. A regular $\mathcal{VB}$-algebroid is called of $\textbf{type $0$}$ (respectively of \textbf{type $1$}), if the core anchor $\partial:C\to E$ is zero (respectively an isomorphism).
\end{definition}

\begin{example}\label{ex:leaf}
Given an arbitrary \vba $(D,A,E,M)$ and a leaf $L$ of $A$, the restriction $(D|_{A_L},A_L,E|_L,L)$ defined by \eqref{eq:restriction} is always regular. Indeed, let $\g^D_e=\ker(\rho_D|_e)$ denote the isotropy Lie algebra of $D$ at $e\in E$. It is easy to check that, for $e=0^E_x$,  $\g^D_{0^E_x}=\g^A_x \oplus \ker(\partial_x)$ for any $x \in M$. Now, for any Lie algebroid, the isotropy Lie algebras have constant rank along its leaves. By noticing that as $x$ varies within $L$,  $0^E_x$ varies inside a fixed leaf of $D$, we see that both $\g^D_{0^E_x}$ and $\g^A_x$ have constant rank over $x\in L$. We thus deduce that the core anchor $\partial_x$ has constant rank when $x$ varies within a leaf $L$ of $A$.
\end{example}

It follows from Corollary \ref{cor: leaves} and Example \ref{ex:leaf} that one is always led to deal with the regular case. 
We start by looking at each case separately.

\subsubsection*{Integrability of type $1$ \vbas}
\label{sub:vb1}By Corollary \ref{int:injectivevba}, a \vba of type $1$ $(D_1,A,E_1,M)$ is integrable if and only if $A$ is integrable as well. In this case, one can even describe the integration explicitly as follows. 

The key observation, made in \cite{GM08}, is that every type $1$ \vba is isomorphic to a particular one defined by $E_1$ and $A$, namely,
\begin{equation}\label{eq:t1:vba}
D_1\simeq \rho_A^*(TE_1),
\end{equation}
where $\rho_A^*(TE_1)$ is given by the pull-back \vba of $(TE_1,TM,E_1,M)$ by the morphism $\rho_A: A \to TM$. Hence, the core of $\rho_A^*(TE_1)$ is $E_1$ and the core-anchor is $\id_{E_1}$. In terms of a splitting $\rho_A^*(TE_1)\simeq D_{A,E_1,E_1}$, a type $1$ \vba structure is equivalent to an $A$-connection $\nabla^E=\nabla^C=\nabla^{E_1}$ on $E_1$ together with (minus) its curvature $\omega=-R_{\nabla^{E_1}}$ (see the Section \ref{subsec:split}).

If $G$ is a Lie groupoid integrating $A$, we denote by $H=({\bold{s}},\tt)^*(E_1\times E_1)$ the pull-back \vbg of  the pair \vbg $(E_1\times E_1, M\times M, E_1, M)$ by the Lie groupoid morphism $({\bold{s}},\tt): G \to M \times M$. Concretely,
\begin{align*}
H &= {\bold{s}}^*E_1 \oplus_G \tt^*E_1\\ &=\left\{ (e,g,e')\in E_1 \times G \times E_1:p_E(e)={\bold{s}}(g), \ p_E(e')=\tt(g)\right\} \rightrightarrows E_1
\end{align*} 
where the source, target and identity maps are $\tilde{{\bold{s}}}(e,g,e')=e$, $\tilde{\tt}(e,g,e')=e'$ and $\tilde{\one}(e)=(e,1_{p_E(e)},e)$; while multiplication is given by
\[ (e,g,e') \cdot (f,h,f') = (f, gh, e').\]
It is not hard to verify that $\LIE(H)\simeq \rho_A^*(TE_1)$ directly from the definitions, and thus:

\begin{proposition}\label{prop:inttype1}
Let $(D_1,A,E_1,M)$ be a \vba of type $1$ and $G$ any integration of $A$. Then, the \vbg $({\bold{s}},\tt)^*(E_1\times E_1)$ integrates $D_1$.
\end{proposition}

Alternatively, the result can be seen as a direct consequence of the fact that the $\LIE$ functor from \vbgs to \vbas commutes with pull-backs since $\LIE({\bold{s}},\tt)=\rho_A$.


\subsubsection*{Integrability of type $0$ \vbas}\label{sub:vba0}
Let $(D_0,A,E_0,M)$ be a \vba of type $0$ with core $C_0$. In this Section, we show that the elements in $\Mon(D_0)\cap \ker(p)$ which obstruct the integrability of $D_0$ can be expressed through an integral formula.

First, we shall follow \cite{GM08} and characterize the structure of $D_0$ in terms of connection data. 
Given a splitting of $D_0$ as in Section \ref{subsec:split} with associated connection data $(\nabla^{E_0},\,\nabla^{C_0},\, \omega_0)$, since the core anchor $\partial$ is zero, the relations in Thm. \ref{lma:conn data} say that both $A$-connections are flat. Moreover, $d_{\nabla^{\scriptscriptstyle \Hom(E_0,C_0)}}$ defines a differential on $\Omega(A;\Hom(E_0,C_0))$ for which $\omega_0$ is a $2$-cocycle. Different choices of  splitting (recall the formulas in Remark \ref{rmk:changesplit}) induce
the same flat $A$-connections $\nabla^{E_0},\,\nabla^{C_0}$ on $E_0$ and $C_0$, while modifying $\omega_0$ by an exact $2$-form in the complex $(\Omega(A;\Hom(E_0,C_0)),d_{\nabla^{\scriptscriptstyle \Hom(E_0,C_0)}})$. In this way, isomorphism classes of type $0$ \vbas correspond to triples $(\nabla^{E_0},\nabla^{C_0}, [\mkern0.5mu\omega_0])$, with  $[\mkern0.5mu\omega_0] \in H^2(A;\Hom(E_0,C_0))$ denoting the class in the underlying cohomology. (See also \cite{GM08}.) 

Our integral formula for the elements in $\Mon(D_0)\cap \ker(p)$ will involve the following ingredients.

\begin{definition}\label{def:spheres}
 An $A$-sphere is a Lie algebroid morphism $\sigma:TI^2 \to A$ with $I=[0,1]$ satisfying $\sigma|_{T\partial I^2} = 0$. The second $A$-homotopy group $\pi_2(A)$ is given by $A$-homotopy classes of such spheres (see \cite{BZ} for further details). The boundary $\partial I^2$ is mapped to a point $x\in M$ called the base point of the sphere.
\end{definition}

Let $W$ be a vector bundle over $M$ and $A\to M$ a Lie algebroid. Suppose that $\nabla$ is a flat $A$-connection on $W$, $\sigma:TI^2 \to A$ an $A$-sphere covering $\gamma:I^2 \to M$ and $\omega \in \Omega^2(A;W)$ a $2$-form on $A$ with values in $W$. Considering the holonomy with respect to the flat pullback $TI^2$-connection $\sigma^*\nabla$ along any path $(0,0) \to (t,s)$ on $I^2$, one can define a trivialization $\tau: \gamma^*W \to I^2 \times W_{\gamma(0,0)}$.
We thus introduce the following integral:
$$
\int_\sigma \omega := \int_{I^2} \tau(\sigma^*\omega) \in W_{\gamma(0,0)}.
$$
Recall that taking holonomy 'flattens' the bundle geometry, namely, it transforms $d_{[\sigma^*\nabla]}$ on $\Omega(I^2;\gamma^*W)$ into de Rham differential $d$ on  $\Omega(TI^2, W_{(0,0)})$:
\begin{equation}\label{eq:flattening}
\tau \circ d_{[\sigma^*\nabla]}  = d \circ \tau.
\end{equation}
We thus have an induced map 

$$\int:\pi_2(A,x) \times H^2(A,W) \to W_x.$$
The integrability problem for type $0$ \vbas can be then summarized in the following way.

\begin{proposition} \label{thm:type0}
Let $(D_0,A,E_0,M)$ be a \vba of type $0$ with underlying class $[\mkern0.5mu\omega_0] \in H^2(A;\Hom(E_0,C_0))$. Then, $D_0$ is integrable if and only if 
$A$ is integrable and the periods of $\omega_0$ vanish, namely, for every $\mkern0.5mu[\mkern0.5mu\sigma] \in \pi_2(A)$, we have that
$$\int_\sigma \omega_0 =0.$$
%
\end{proposition}
\begin{proof}
The proof consists in showing that, for any $e\in E_0$, we have
 $$\Mon(D_0)_e \cap \ker (p) = \set{\mkern0.5mu\int_\sigma \omega_0(e): [\mkern0.5mu\sigma] \in \pi_2(A,p_E(e))},$$
 so that the claim follows from Theorem \ref{prop:intiff}. 
 Consider an element $v \in \Mon(D_0)_e \cap \ker( p)$. By definition, there exist a $D_0$-homotopy $h:TI^2 \to D_0$  between the constant $D_0$-paths $vdt$ and $0_E^D(e)dt$. The map $\sigma:=p\circ h : TI^2\to A$  then defines an $A$-sphere by direct inspection of the boundary conditions $h|_{T\partial I^2}$ and we denote by $\gamma:I^2\to M$ the base map of $\sigma$.
  We want to show that $$v=\int_\sigma \omega_0(e).$$ To that end, consider a splitting $D_0 \simeq D_{A,E_0,C_0}$ with associated connection data $(\nabla^{E_0},\nabla^{C_0},\omega_0)$. In the split \vba, the homotopy $h$ takes the form \[ h \equiv (\sigma, \hat{e}, \psi): TI^2 \to A\times_M E \times_M C,\]
  where $\hat{e}: I^2 \to E$ is a map covering $\gamma$ which can be seen as a section of $\gamma^*E$ and $\psi:TI^2\to C$ can be seen as a $1$-form on $I^2$ with values on $\gamma^*C$. 
    Let us split the condition that $h$ is a Lie algebroid morphism into conditions for $\hat{e}$, $\sigma$ and $\psi$. First, $\sigma$ must be a Lie algebroid morphism covering $\gamma$. Secondly, recall that the induced pullback map $h^*: \Gamma(E_0,\wedge^\bullet D^*_{E_0}) \to \Omega^\bullet(TI^2)$ between forms must commute with differentials. A computation shows that this implies (recall eq. \eqref{eq:Drep} and Remark \ref{rmk:CEdiff})
  \[ d_{[\sigma^*\nabla^{E_0}]}\hat{e}= 0 , \ \ \ d_{[\sigma^*\nabla^{C_0}]}\psi = \sigma^*\omega_0 (\hat{e}), \] 
  where $\sigma^*\nabla^{E_0}$ and $\sigma^*\nabla^{C_0}$ denote the (flat) pullback connections on $\gamma^*C_0$ and $\gamma^*E_0$, respectively. Using parallel transports $\tau^{E_0}$ and $\tau^{C_0}$ with respect to these flat connections we can trivialize $\gamma^*C_0$ and $\gamma^*E_0$ to $I^2\times E_0|_{\gamma(0,0)}$ and $I^2\times C_0|_{\gamma(0,0)}$, respectively. Then, the flat section $\hat{e}$ gets transformed into the constant section with value $\hat{e}(0,0)=e$ and  $\psi$ gets transformed into a $1$-form $\psi_0$ on $I^2$ with values in the vector space $C|_{p_E(e)}$ which, because of eq. \eqref{eq:flattening}, satisfies the equation
   \[d\psi_0 = \tau^{C_0} \circ \sigma^*\omega_0 \circ \tau^{E_0}(e).\] Integrating both sides over $I^2$ and using Stokes' theorem together with the boundary conditions (the only non-trivial boundary values being $\psi_0(\partial_t)|_{(t,0)}=v$ for any $t\in I$), we get the desired formula for $v$.
   
   Conversely, given $\sigma: TI^2\to A$ an $A$-sphere one can define $\hat{e}$ by parallel transport of $e$ and $\psi$ as above with $\psi_0 \in \Omega^1(I^2;C|_{\gamma(0,0)})$ given by
   \[\psi_0 = \left( \int_0^1 \int_s^1 g(u,r) du dr\right) dt + \left(\int_0^t g(u,s) du - t \int_0^1 g(u,s) du\right) ds,\]
   where we have written $\tau^{C_0}(\sigma^*\omega_0)(\hat{e})=g(t,s) dt\wedge ds$ for $g:I^2 \to C_0|_{\gamma(0,0)}$. Evaluating at $s=0,1$, we see that $h\equiv(\sigma,\hat{e},\psi)$ defines a $D$-homotopy between the constant $D$-path $vdt$ with $v=\int_\sigma \omega_0(e)$ and the zero path $0^D_Edt$, so that $v\in \Mon(D_0)_e\cap \ker(p)$.
  \end{proof}

Let us give a cohomological interpretation of the integrability criteria.
Assume that the flat $A$-connections can be integrated to representations of an integration $G$ of $A$ on both $E_0$ and $C_0$. In such a case, there is an induced representation of $G$ on $\Hom(E_0,C_0)$ and, to answer the question of how to lift $\omega_0$ from $A$ to $G$, one is lead to consider the \textbf{Van Est map}
 $${\VV}:H^\bullet(G,\Hom(E_0,C_0)) \longrightarrow H^\bullet(A,\Hom(E_0,C_0)).$$
 (More details will be given in Examples \ref{ex:vbg0_1},\ref{ex:vbg0_2} below.)
The Van-Est map for Lie groupoids and algebroids was studied in \cite{C00}. If $G$ is source simply-connected, then the representations $\nabla^{E_0}$ and $\nabla^{C_0}$ always integrate to representations of $G$ on $E_0$ and $C_0$, respectively. Moreover,  in this case, the Van Est map between degree $2$-cohomologies is \emph{injective}. 
For an element in $H^2(A,\Hom(E_0,C_0))$ to be in the image of $\VV$, the criteria given in \cite[Cor. $2$]{C00} is exactly the vanishing of the corresponding spherical periods as in Prop. \ref{thm:type0}.

 \begin{proposition}\label{prop:VE0} A type $0$ \vba $(D_0,A,E_0,M)$ is integrable iff $A$ is integrable and $[\mkern0.5mu\omega_0]$ lies in the image of the Van Est map, namely:\vspace{-5pt}
 $$\exists\mkern1mu \hat{\omega}_0 \in H^2(\G(A),\Hom(E_0,C_0))\text{ such that }\VV(\hat{\omega}_0)=[\omega_0].$$
 \end{proposition}

\begin{remark}\label{rmk:H0}
In the case when $D_0$ is integrable, there is a way of constructing an integrating \vbg by integrating each piece in the connection data. Namely, let $G$ be an integration of $A$ such that $\nabla^{E_0}$ and $\nabla^{C_0}$ integrate to representations of $G$ on $E_0$ and $C_0$. Furthermore, let $\hat{\omega}_0 \in H^2(G,\Hom(E_0,C_0))$ be such that $\VV(\hat{\omega}_0)=[\omega_0]$.
In this case, the data given by the representations of $G$ on $E_0$ and $C_0$ together with any representative of the class $\hat{\omega}_0$ can be used to define a (type 0) \vbg that we shall denote $(H_{0,\hat{\omega}_0},G,E_0,M)$. In Section \ref{sec:reputh}, Examples \ref{ex:vbg0_1}, \ref{ex:vbg0_2}, we provide an explicit definition of this \vbg and show how it defines an integration of the given type $0$ \vba.
\end{remark}

\subsubsection*{Integrability of general regular type \vbas}
Consider a \vba $(D,A,E,M)$ of regular type. The key point for us is the following result of \cite{GM08} saying that such a \vba can be decomposed into type $0$ and type $1$ components.

\begin{theorem}[\cite{GM08}]\label{thm:typedec}
  Let $(D,A,E,M)$ be a \vba of regular type. There exists (unique up to isomorphism) \vbas $(D_0,A,E_0,M)$ and $(D_1,A,E_1,M)$ respectively of type $0$ and type $1$, such that $D\simeq D_0 \oplus_A D_1$.
\end{theorem}

The decomposition of Theorem \ref{thm:typedec} is based on decomposing the core-anchor map $\partial:C \to E$ into a direct sum of a trivial map and an isomorphism. Namely, one uses isomorphisms $C \simeq \ker \partial \oplus \im \partial$ and $E \simeq \coker \partial \oplus \im \partial$ so that $\partial \simeq 0 \oplus \id_{\im\partial}$. The proof consists in showing that these isomorphisms can always be lifted to the full \vba structure by means of an appropriate choice of splitting. 

The type $0$ component $D_0$ can be thus assumed to have core $C_0=\ker \partial$ and side $E_0=\coker \partial$. Recall from Remark \ref{rmk:repincohom}, that any set of connection data $(\nabla^E,\nabla^C, \omega)$ for $D$ defines flat $A$-connections $\nabla^{\ker \partial}$ and $\nabla^{\coker\partial}$  on $\ker \partial$ and $\coker \partial$, respectively. Moreover, using the decompositions of $E$ and $C$ into type $0$ and type $1$ components, one can define the projection  $\omega_0 \in \Omega^2(A; \Hom(\coker \partial,\ker \partial))$ of $\omega$, which turns out to be a cocycle for the induced differential $d_{\nabla^{\scriptscriptstyle \Hom(\coker\partial, \ker\partial)}}$. Because of the change of splitting formulas of Remark \ref{rmk:changesplit}, the data $(\nabla^{\im\partial},\nabla^{\ker\partial},[\mkern1mu\omega_0])$ characterizing the isomorphism class of the type $0$ component is independent of the initially chosen splitting and it is thus intrinsically associated to the regular \vba $D$. 

\begin{proposition} \label{prop:reg type}
Let $(D,A,E,M)$ be a \vba of regular type. Then the following assertions are equivalent:
\begin{enumerate}[(i)]
\item $D$ is integrable,
\item $D_0$ is integrable,
\item $A$ is integrable and the periods of $\omega_0\in\Omega^2(A,\Hom(\coker \partial,\ker\partial))$ vanish, namely:
  $$\ 
\int_\sigma \omega_0 =0, \text{ for every } \mkern0.5mu[\mkern0.5mu\sigma] \in \pi_2(A),$$
\item $A$ is integrable and $[\mkern0.5mu\omega_0]$ lies in the image of the Van Est map, namely:\vspace{-1,7mm}
 $$\text{there exists }\mkern1mu \hat{\omega}_0 \in H^2\bigl(\G(A),\Hom(\coker \partial,\ker\partial)\bigr)\text{ such that }\,\VV(\hat{\omega}_0)=[\omega_0].$$\end{enumerate}
\end{proposition}

\begin{proof}
This follows from  Theorem \ref{thm:typedec}, Propositions \ref{prop:inttype1}, \ref{thm:type0} and \ref{prop:VE0}, and from the fact that the functor $\LIE$ preserves direct sums.
\end{proof}

\begin{remark}
 In the integrable case, an integrating \vbg can be given as $H:=(s,t)^*(E_1 \times E_1)\oplus_G H_{0,\hat{\omega}_0}$ where each summand was explained in Proposition \ref{prop:inttype1} and Remark \ref{rmk:H0}, respectively.
\end{remark}

\subsection{Computing the obstructions}        \label{sub:comp:obs}             
So far, to compute the obstructions to the integrability of a \vba $D$, we must pick a leaf $L$ of $A$, extract the type $0$ part of the regular \vba $D|_{A_L}$ and then compute the corresponding spherical integrals. In this section, we show that there is a more practical formula for the obstructions only involving the connection data associated to \emph{any global splitting} of $D$.

Consider a \vba $(D,A,E,M)$ together with a splitting, and denote by $(\nabla^E,\nabla^C,\omega)$ the corresponding connection data.
 Given an $A$-homotopy $\sigma = adt+bds:TI^2\to A$, we shall denote $\gamma_t^s$ rather than $\gamma(t,s)$ the base map. Similarly, for any $t,t'\in[0,1]$, we will denote $a^s_{t',t}:T[t,t']\to A$ the $A$-path $a|_{[t',t]\times\{s\}}dt$, while:
\begin{align*}\hol^E_{a^s_{t'\!\!,t}}&:E_{\gamma^s_t}\to E_{\gamma^s_{t'}}, \\
                \hol^C_{a^s_{t'\!\!,t}}&:C_{\gamma^s_t}\to C_{\gamma^s_{t'}}
\end{align*}
will denote the corresponding holonomies. We also associate to $\sigma$ a map $F_\sigma: E|_{\gamma(0,0)} \to C|_{\gamma(1,0)}$ defined by the integral formula
\begin{equation}\label{eq:Fsigm}F_\sigma(e)= \int_0^1\!\!\!\int_0^1 \hol^C_{a_{1,t}^s}\!\circ\ \omega_{\gamma_t^s}(a,b)\circ \hol^E_{a^s_{t,0}}(e) \ dtds\in C|_{\gamma(1,0)},  \ \ e\in E|_{\gamma(0,0)}.\end{equation}
When an $A$-connection is flat, the corresponding holonomy only depends on the $A$-homotopy class of the underlying $A$-path. For the connection data $(\nabla^E,\nabla^C, \omega)$ we shall show that a generalized (or \emph{homotopy}) version of this fact holds.
The first step is to recall the following lemma which relates the holonomy and the curvature of an $A$-connection by an integral formula.
\begin{lemma}\label{lem:curvature:global}
For an arbitrary $A$-connection $\nabla$ on a vector bundle $V\to M$ and any $A$-homotopy $adt+bds:TI^2\to A$, the       
following relation holds:                                                                                                                  
\begin{equation}\label{eq:curvature:global}                                                                                                                                   
\frac{d}{ds} \hol^V_{a_{1,0}^s}(x)=\int_0^1 \hol^V_{a^s_{1,t}}\circ\ R_\nabla(a,b)_{\gamma_{t}^{s}} \circ \hol^V_{a^s_{t,0}}(x)dt,                                            
\end{equation}
where $R_\nabla\in\Omega^2(A,\End(V))$ denotes the curvature of $\nabla$.                                                                                                                                                                
\end{lemma}
This result  goes back to \cite{NI01} in the case of usual linear connections (\emph{i.e.} $TM$-connections)  and can be proved by a simple argument of variation of parameters. The case of an arbitrary $A$-connection follows by pulling back the $A$-connection along $\sigma:TI^2\to A$.

\begin{lemma}\label{lma:Fsigma}
 Let $(\nabla^E,\nabla^C, \omega)$ denote connection data associated to $D$ and $\sigma = adt + b ds: TI^2 \to A$ an $A$-homotopy as above. Then,
 \[ \hol^E_{a^1_{1,0}} - \hol^E_{a^0_{1,0}} =-\partial \circ F_\sigma , \ \ \hol^C_{a^1_{1,0}} - \hol^C_{a^0_{1,0}} = -F_\sigma \circ \partial.\]
 \end{lemma}

 \begin{proof}
 This is a direct consequence of Lemma \ref{lem:curvature:global}. For instance, taking $e\in E_{\gamma(0,0)}$, we compute that:
\begin{align*}
\partial F_\sigma(e) &= \ \int_0^1\!\!\!\int_0^1 \partial\circ \hol^C_{a_{1,t}^s}\!\circ\ \omega(a,b)_{\gamma_t^s}\circ \hol^E_{a^s_{t,0}}(e)dt ds \\
&= \int_0^1\!\!\!\int_0^1  \hol^E_{a_{1,t}^s}\!\circ\ \partial\, \omega(a,b)_{\gamma_t^s}\circ \hol^E_{a^s_{t,0}}(e)dt ds \\
&= -\int_0^1\!\!\!\int_0^1  \hol^E_{a_{1,t}^s}\!\circ\ R_{\nabla^E}(a,b)_{\gamma_t^s}\circ \hol^E_{a^s_{t,0}}(e)dt ds \\
&= -\int_0^1 \frac{d}{ds} \hol^E_{a_{1,0}^s}(e) ds \\
&= \hol^E_{a^0_{1,0}}(e) - \hol^E_{a^1_{1,0}}(e).
\end{align*}
Here we used the relations in Theorem \ref{lma:conn data} to, first, commute $\partial$ and $\hol^E$ and, second, to obtain the curvature of $\nabla^E$ and apply Lemma \ref{lem:curvature:global}. The proof of the second statement is analogous.
 \end{proof}

\begin{remark}
 There is a clear interpretation of the above Lemma in terms of the $2$-term complex $\partial:C \to E$ (see also \cite[Prop. 3.13]{AriasCrainic}). Namely, an $A$-path $a$ defines a chain morphism $\Hol_a=(\hol^C_{a_{1,0}},\hol^E_{a_{1,0}})$ while an $A$-homotopy $\sigma$ defines a chain homotopy $F_\sigma: E_{\gamma(0)} \to C_{\gamma(1)}$ between the  $\Hol_{a^{s=0}}$ and $\Hol_{a^{s=1}}$:
 $$\SelectTips{cm}{}\xymatrix{ C_{\gamma(0)} \ar[d]_{\hol^C_{a^{s=0,1}}} \ar[r]^{\partial} & E_{\gamma(0)} \ar[d]^{\hol^E_{a^{s=0,1}}}\ar@{-->}@<1ex>[dl]_{F_\sigma}\\
             C_{\gamma(1)}                   \ar[r]_{\partial} & E_{\gamma(1)}. }$$
\end{remark}

Let us go back to our integrability problem. An $A$-sphere $\sigma=adt+bds: TI^2\to A$ based at $m\in M$ can be seen as an $A$-homotopy between the trivial $A$-path $0^A(m)dt$ and itself. The resulting map $F_\sigma \in \Hom(E_m,C_m)$ will be denoted:
\begin{equation}\label{eq:int:formula}
\int_\sigma^{\nabla} \!\omega(e):=  \int_0^1\!\!\!\int_0^1 \hol^C_{a_{1,t}^s}\!\circ\ \omega(a,b)_{\gamma_t^s}\circ \hol^E_{a^s_{t,0}}(e)dtds\in C_m, \ \ e\in E_m,\end{equation}
and will be called the \textbf{period of} $\omega$ \textbf{along} $\sigma$.
We can now relate the obstructions to the integrability of $D$ to the periods of the globally defined $\omega$. To that end, let $L$ be the leaf of $A$ such that the $A$-sphere $\sigma$ lies in $A_L\hookrightarrow A$.
Since the restriction $D|_{A_L}$ is of regular type, we can consider a decomposition:
$$\quad \quad\mathcal{T} : D|_{A_L} \longrightarrow \overline{D}_L= \overline{D}_{0,L} \oplus_A  \overline{D}_{1,L}, \quad \quad \quad \mathcal{T}=\mathcal{T}_0\oplus\mathcal{T}_1,$$
where $\overline{D}_{i,L}$ is of type $i=0,1$. Let
\begin{align*} 
\quad\quad\mathcal{T}^C:  C|_L  &\longrightarrow  \overline{C}_L =  \ker \partial \oplus  \overline{C}_{1,L},\quad\quad\quad \mathcal{T}^C=\mathcal{T}_0^C\oplus \mathcal{T}_1^C, \\
\quad\quad\mathcal{T}^E:  E|_L &\longrightarrow \overline{E}_L =  \coker\partial \oplus  \im\partial ,\quad\quad\quad \mathcal{T}^E=\mathcal{T}_0^E\oplus \mathcal{T}_1^E,
\end{align*}
be the isomorphisms induced by $\mathcal{T}$. We shall denote by
 $$\overline{\omega}_{0,L} = \in   \Omega^2(A_L, \Hom(\coker \partial,\ker \partial))$$
 the projection of $\omega$ to the type $0$ factor defined by $\overline{\omega}_{0,L} = \mathcal{T}_0^C \circ \omega|_{A_L} \circ  (\mathcal{T}^E)^{-1}|_{\coker\partial}$. 
 Then, Lemma \ref{lma:Fsigma} implies that 
\begin{align*}
\mathcal{T}^C\left(\int_\sigma^{\nabla} \!\omega(e)\right)= \int_{\sig}  \overline{\omega}_{0,L}(\overline{e}_0) \oplus 0 \in  \ker{\partial} \oplus \overline{C}_{1,L},
\end{align*}
for any $A_L-$sphere $\sigma$, $e \in E|_L$ and $\overline{e}_0:=\mathcal{T}^E_0(e)$. This is so because the holonomy associated to the trivial $A$-path is trivial and, hence, $F_\sigma \circ \partial =0$ and $\partial \circ F_\sigma =0$ so that $F_\sigma$ takes the above form using the decomposition $\mathcal{T}$. Notice that, since $\overline{\omega}_{0,L}$ is a cocycle, the period of $\omega$ along $\sigma$ depends only on the homotopy class of $\sigma$ as in the type $0$ case.

By applying Corollary \ref{cor: leaves} and Proposition \ref{prop:reg type},  we see that \eqref{eq:int:formula} gives an integral expression for the obstructions in terms of a global splitting of $D$. Therefore, one can state the following:

\begin{theorem}\label{thm:integrals}
A \vba $(D,A,E,M)$ is integrable if and only if  $A$ is integrable and,  for any connection data $(\nabla^E,\nabla^C,\omega)$ induced by a splitting of $D$, the periods of $\omega$ vanish, namely:
$$ \quad \int^{\nabla}_\sigma \omega=0, \ \ \forall \sigma \in \pi_2(A).$$
\end{theorem}

\begin{remark}
A \vba together with a splitting can be seen as a special case of a Lie algebroid extension
 $$\SelectTips{cm}{}\xymatrix@C=15pt{\ker p\ \ar@{^{(}->}@<-1pt>[r] & D\ar@<-1pt>@{->>}[r]& A,}$$ 
 with a complete Ehresmann connection in the sense of \cite{Br}. In \cite{BZ} it is shown that there is an associated transgression map $\delta:p_E^*\pi_2(A)\to \G(\ker p)$ that fits into a long homotopy exact sequence and which is related to the integrability of the fibration (\cite{Br}). A careful inspection shows that, indeed, the integral formula in Theorem \ref{thm:integrals} coincides with the transgression map of \cite{BZ}. We shall expand on this aspect of things in a future work.
\end{remark}

We conclude this section with an explicit example of a non-integrable \vba $(D,A,E,M)$ in which $A$ is integrable.

\begin{example}[A non integrable \vba]\label{ex:non:int}
\ale Consider $E=S^2\times \mathbb{R}$ and $C=S^2\times \mathbb{R}$, both seen as a trivial vector bundles over $M=S^2$. 
Let $A:=TS^2$ be the tangent algebroid of $S^2$, and consider the  split double vector bundle $D=\Dd$, so that
$$\SelectTips{cm}{}
D=\xymatrix@R=20pt@C=25pt{ *+[c]{TS^2\times\mathbb{R}\times\mathbb{R}}\ar[d]_-{}\ar[r]^-{ p} &*+[c]{TS^2}\ar[d]^{ }\\
           *+[c]{S^2\times \mathbb{R}}\ar[r]^{ }                          &*+[c]{S^2.}}$$
We define a \vba structure on $D$ by setting:
\begin{align*}
&&&&[X,Y]_D&=[X,Y]_{TS^2}+\bar{e} \ \omega_0(X,Y)\ \cc, & \rho_D(X)&=X, &&&&\\
&&&&[X,\cc\mkern1mu]_D&=0,                                     & \rho_D(\mkern1mu\cc\mkern1mu)&=0.&&&&
\end{align*}
Here $X,Y\in \Gamma(TS^2)$ are vector fields on $S^2$, $\bar{e}:E \to \R$ is the fiberwise linear coordinate function, $\omega_0$ denotes the standard symplectic structure on $S^2$ and $\cc\in \Gamma_c(E, D)$ denotes the core section corresponding as in \eqref{core_section} to a non-vanishing constant section $c\in \Gamma(C)$. The homogeneous structure for $D\to A$ is defined by $m_\lambda(a,e,c)=(a,\lambda e ,\lambda c)$.

 It is easy to verify that this defines the structure of a \vba $(D,TS^2,E,S^2)$ whose core-anchor is zero, i.e., that $D$ is a \vba of type $0$. Observe that $h(X)=X$ defines a horizontal lift whose associated $A$-connections are trivial and whose  curvature is given by:
$$\omega = \omega_0 \otimes (\bar{e} \otimes c)\in \Omega^2(S^2,\Hom(E,C)).$$			
 Theorem \ref{thm:integrals} can be directly applied to conclude that $D$ is not integrable. Namely, an $A$-sphere $\sigma:TI^2\to TS^2$ must be the derivative of its base map $\gamma: I^2 \to S^2$ (which collapses $\partial I^2$ to a point) and the integral \eqref{eq:int:formula} reduces to the standard integral
\[ \int^\nabla_\sigma \omega(e) = \bar{e}(e) \left[\int_{I^2} \gamma^*\omega_0\right] \ c|_{p_E(e)}, \]
which will not vanish for all spheres $\sigma$ since $\omega_0$ is a volume form.
\end{example}

\begin{remark}
 The above non-integrable Example \ref{ex:non:int} has a very similar structure to that of the previous integrable Example \ref{ex:int}. In the integrable Example \ref{ex:int}, though, the spherical periods of our Theorem \ref{thm:integrals} vanish automatically because $\pi_2(A)=0$. Indeed, this happens because $G=T^*SU(2)\rightrightarrows \su$ integrates $A=T^*_{\pi_M}\su$ and, then, the second homotopy group  of $A$ coincides with the one of the source fibers of $G$, namely, with $\pi_2(SU(2))=0$ (see also \cite{CF03, C00,BZ}).
%
\end{remark}

\section{Integrating $2$-term representations up to homotopy}\label{sec:reputh}
In this section, we study Lie theory of $2$-term representations up to homotopy by translating the results already obtained for \vbas by means of the equivalence of \cite{GM08}, as recalled in Section \ref{subsec:split}.

\subsection{Representations up to homotopy of Lie groupoids} 
\ale
Recall from Section \ref{subsec:split} that $2$-term representations up to homotopy of $A$ on $\E=C\oplus E$ correspond to  \vba structures $(\partial,\nabla^E,\nabla^C,\omega)$ on the split \DVB $\Dd = A\times_M E \times_M C$. Fixing the underlying algebroid $A$, we thus have an assignment
\[\VBA: \Rep(A)\to  \VBA(A)\]
where $\Rep(A)$ denotes the category of $2$-term representations up to homotopy of $A$ and $\VBA(A)$ the category of \vbas of the form $(D,A,E,M)$. Notice that the image of $\VBA$ consists of split \vbas. The notion of morphism between representations up to homotopy can be found in \cite{AriasCrainic} (see also \cite{GM08}) and $\VBA$ turns out to be a functor (see \cite{DJO}). In particular, two elements of $\Rep(A)$ are isomorphic if the corresponding split \vbas are.

In this subsection, we follow \cite{GM10} and introduce the analogue of the above correspondence for Lie groupoids. 

First, we introduce the notion of a $2$-term representations up to homotopy of a Lie groupoid $G$ over $M$. 
To that end, recall that any vector bundle $E\to M$ determines a smooth category $L(E)$ whose objects are elements of $M$, and morphisms between $x,y\in M$ are linear maps $E_x\to E_y$ (not necessarily invertible). We say that a smooth map $\Delta: G\to L(E)$ is a \textbf{quasi-action} of $G$ on $E$ if $\Delta$ preserves both source and target maps. A quasi-action is said to be \textbf{unital} if it preserves the units (i.e. it is the identity on objects) and \textbf{flat} if it preserves composition. Notice that a representation of $G$ on $E$ is just a flat unital quasi-action $\Delta:G\to L(E)$ such that $\Delta_g:E_{{\bold{s}}(g)}\to E_{\tt(g)}$ is invertible for each $g\in G$. 

\begin{definition}\label{def:ruthgroupoids}

Let $G$ be a Lie groupoid over $M$. A \textbf{representation up to homotopy} of $G$ on the graded vector bundle $\E=C\oplus E$, is given by a quadruple $(\partial,\Delta^C,\Delta^E,\Omega)$ where:
\begin{itemize}
\item $\partial:C\to E$ is a bundle map, 
\item $\Delta^C$ and $\Delta^E$ are unital quasi-actions on $C,E$, respectively, 
\item $\Omega\in \Gamma (G^{(2)}, \Hom({\bold{s}}_{(2)}^*E,\tt_{(1)}^*C))$ is a section assigning to each composable pair $(g_1,g_2)\in G^{(2)}$ a linear map $\Omega_{g_1,g_2}:E_{{\bold{s}}(g_2)} \to C_{\tt(g_1)}$ which is normalized (i.e. $\Omega_{g_1,g_2}=0$ if either $g_1$ or $g_2$ is a unit), 
\end{itemize}
satisfying the following conditions:

\begin{align}
\Delta^E_{g_1}\circ \partial&=\partial\circ \Delta^C_{g_1}\nonumber\\
\Delta^C_{g_1}\Delta^C_{g_2}- \Delta^C_{g_1g_2} &+ \Omega_{g_1,g_2}\partial=0\nonumber\\
\Delta^E_{g_1}\Delta^E_{g_2}- \Delta^E_{g_1g_2} &+ \partial\Omega_{g_1,g_2}=0\nonumber\\
\Delta^C_{g_1}\Omega_{g_2,g_3}-\Omega_{g_1g_2,g_3}&+\Omega_{g_1,g_2g_3}-\Omega_{g_1,g_2}\Delta^E_{g_3}=0 \label{eq:repGeqs}
\end{align}
for every composable triple $(g_1,g_2,g_3)$ in $G$.

\end{definition}

In other words, a representation up to homotopy of $G$ on $\E=C\oplus E$ is given by unital quasi-actions on $C$ and $E$, which are not necessarily flat, but the flatness of those is controlled by the 2-cocycle $\Omega$. One can also think that this notion corresponds to that of a representation of $G$ on the $2$-term complex \[C\overset{\partial}{\to} E.\] 

\begin{remark}
Representations up to homotopy of $G$ on an arbitrary graded vector bundle $\E$ can be defined in cohomological terms, as a differential in the complex $C(G;\E)$ of $\E$-valued groupoid cochains. This notion can be found in \cite{AriasCrainic2} for arbitrarily graded $\E$ and the equivalence to the above definition in the $2$-term case, in \cite{GM10}. 
\end{remark}

Following \cite{GM10} further, there is a $1:1$ correspondence between $2$-term representations up to homotopy of $G$ on $\E=C\oplus E$ and  \vbg structures on $(\tt^*C\oplus_G {\bold{s}}^*E,G,E,M)$, as follows. We denote the elements of $\tt^*C\oplus_G {\bold{s}}^*E$ as triples $(c,g,e)$ with $c\in C_{\tt(g)}$, $g\in G$ and $e\in E_{{\bold{s}}(g)}$. The source, target, identity and multiplication maps are then given by
\begin{align}
 \tilde{\bold{s}}(c,g,e)=e, \quad \tilde{\tt}(c,g,e)=\partial(c)+\Delta^E_g(e),\quad \tilde{\one}_e=\left(0^C(x),\one_x, e\right), \ e\in E_x \nonumber \\
(c_1,g_1,e_1)\cdot (c_2,g_2,e_2)=\left(c_1+\Delta^C_{g_1}(c_2)-\Omega_{g_1,g_2}(e_2),g_1\cdot g_2,e_2\right). \label{eq:splitvbgmaps}
\end{align}
Such \vbg structures will be called \emph{split}. 

\begin{remark}[\vbg splittings]\label{rmk:splitG}
\vbgs can be split in a similar way to the splitting of \vbas described in Section \ref{subsec:split}. For any \vbg $(H,G,E,M)$, there is an associated short exact sequence of vector bundles over $G$:
$$\SelectTips{cm}{}\xymatrix@C=17pt{ V^R \ \ar@{^{(}->}@<-0.35pt>[r] & H  \ar@{->>}[r]<-0.4pt> &  {\bold{s}}^*E,}$$
where the map on the right is $\gamma \mapsto (q_H(\gamma),\tilde{s}(\gamma))$.
The vector bundle $C \to M$ given by $C=\one^* V^R$ is called the {\bf right core} of $H$ and multiplication by zero on $H$ can be used to identify $V^R\simeq \tt^*C$. A {\bf right splitting} of $(H,G,E,M)$ is a splitting $h_G$ of the above short exact sequence such that $h_G(\one_x, e) = \tilde{\one}_e$, for all $x\in M$, $e\in E_x$.
 The corresponding vector bundle isomorphism $H \simeq \tt^*C \oplus_G {\bold{s}}^*E$ induces a \emph{split} \vbg structure on $(\tt^*C\oplus_G {\bold{s}}^*E,G,E,M)$ in the sense defined above. Thus, each splitting $h_G$ for $H$ determines a $2$-term representation up to homotopy of $G$ on $C\oplus E$. The reader is referred to \cite{GM10} for further details.
\end{remark}

Denoting $\Rep(G)$ the category of $2$-term representations up to homotopy of $G$ and $\VBG(G)$ the category of \vbgs with base $G$, we get an assignment \[\VBG: \Rep(G) \to \VBG(G),\] whose image consists of split \vbgs.

\begin{example}\label{ex:vbg0_1}
 Let $(\partial,\Delta^E,\Delta^C,\Omega)$ be a representation up to homotopy of $G$. When the map $\partial$ is zero, eqs. \eqref{eq:repGeqs} imply that $\Delta^E$ and $\Delta^C$ define ordinary representations of $G$ on $E$ and $C$, respectively. Then, $\Hom(E,C)$ also inherits a representation of $G$ given by
 \[ \Hom(E,C)|_{{\bold{s}}(g)} \ni \phi \mapsto \Delta^C_g \circ \phi \circ \Delta^E_{g^{-1}} \in \Hom(E,C)|_{\tt(g)}.\]
 Moreover, one can define a $2$-cochain $$\hat{\omega} \in C^2(G;\Hom(E,C)) = \Gamma(G^{(2)}, \tt_{(2)}^*\Hom(E,C))$$ in groupoid cohomology with coefficients in $\Hom(E,C)$ out of $\Omega$ by:
 \[ \hat{\omega}_{g_1,g_2} : = \Omega_{g_1,g_2} \circ \Delta^E_{g^{-1}_2g^{-1}_1}.\]
 The last equation in \eqref{eq:repGeqs} is then equivalent to requiring $\hat{\omega}$ being a cocycle in $C^2(G;\Hom(E,C))$ in the sense of, e.g., \cite{C00}. The corresponding split \vbg is the analogue of the type $0$ \vba of Section \ref{sub:vba0}, and we get that these are analogously characterized by representations of $G$ on $E$ and $C$ together with a cohomology class in $H^2(G;\Hom(E,C))$. (See also \cite{GM10}).
\end{example}

\subsection{Lie theory for $2$-term representations}
Let $G$ be a Lie groupoid with Lie algebroid $A$. By means of the correspondence between representations up to homotopy and \vbas/\vbgs, we shall define a differentiation operation $\Lie:\Rep(G) \longrightarrow \Rep(A)$ by the commutative diagram
 \begin{align*}
\SelectTips{cm}{}
\xymatrix{ \Rep(G) \ar[d]_{\Lie}\ar[r]^{\VBG} & \VBG(G) \ar[d]^{\LIE}\\
           \Rep(A) \ar[r]^{\VBA}& \VBA(A)}
\end{align*}
and we shall write $\Lie(\partial_G,\Delta^E,\Delta^C, \Omega) = (\partial_A,\nabla^E,\nabla^C,\omega)$.

To that end, we need to verify that this assignment $\Lie$ is well defined, namely, we must check that \emph{the \vba associated to a split \vbg by the $\LIE$ functor is naturally split}. 

To see this, consider a split \vbg structure on $\Gamma=\tt^*C\oplus {\bold{s}}^*E$. Let $c\in C_x$, $e\in E_x$ and $g(\e)\in G$ a curve such that ${\bold{s}}(g(\e))=x$  for any $\e,$ and $g(0)=1(x)$, so that $\dot{g}(0)=a$ defines a Lie algebroid element in $A_x=\ker(T{\bold{s}})_{\one_x}$. Using the notation of equations \eqref{eq:splitvbgmaps}, the curve
\[ \e \mapsto \gamma(\e)=(\e\Delta^C_{g(\e)}c, g(\e), e) \in \Gamma\]
is contained in the $\tilde{{\bold{s}}}$-fiber and starts at $\gamma(0)= (0^C,\one_x,e) =\tilde{\one}_e$. Then, $\dot{\gamma}(0) \in \LIE(\Gamma)_{e}$ defines a Lie algebroid element which, under the identifications $T\Gamma\simeq(T\tt)^*TC\oplus_{TG}(T{\bold{s}})^*TE$ and $TC|_{0^C} \simeq C \oplus TM$, can be written as
\[ \dot{\gamma}(0) = (c\oplus \rho(a),a,0) \in \LIE(\Gamma)_e\subset T_{(0,\one_x,e)}\Gamma.\]
We thus get the desired splitting map 
\begin{align*}
i_{\Gamma}:\Dd \to &\ \LIE(\Gamma),\\ 
(a,e,c)\mapsto&\  (c\oplus \rho(a),a,0)|_{(0,\one_x,e)}.
\end{align*}
It also follows from the previous argument that the core anchor is preserved by $\Lie$, namely, that $\partial_G=\partial_A$. Indeed, the anchor map of $\LIE(\Gamma)\simeq \Dd$ is given by $\rho_D=T\tilde{\tt}\circ i_{\Gamma}$, so we obtain from \eqref{eq:splitvbgmaps}:
$$\rho_D(a,e,c)=\left.\frac{d}{d\e}\right|_{\e=0} [\Delta^E_{g(\e)} (e + \e \partial_G c)]\in T_eE$$
where $a = \dot{g}|_{\e=0}$ as before. Thus, the induced core anchor is given by $\partial_A(c):=\rho_D(0^A,0^E,c)\equiv \partial_G(c)$ where we used the identification of $E$ as the core of $TE$ given in Example \ref{ex:TE1}.

For the rest of the structure maps, we get the following formulas.
\begin{lemma}
 Let $(\partial,\Delta^E,\Delta^C,\Omega)\in \Rep(G)$ a $2$-term representation up to homotopy of $G$, and  $(\partial,\nabla^E,\nabla^C, \omega)=\Lie(\partial,\Delta^E,\Delta^C,\Omega)\in \Rep(A)$ the induced representation of $A$. Then the following relations hold:
  \begin{alignat*}{3}
   \nabla^E_\alpha e (x) &=& \left.\frac{d}{d\e}\right|_{\e=0}&\  \Delta^E_{(\phi^\alpha_\e(x))^{-1}} (e\circ \tt\circ\phi^\alpha_\e(x)) \\
   \nabla^C_\alpha c (x) &=&\left. \frac{d}{d\e}\right|_{\e=0}& \ \Delta^C_{(\phi^\alpha_\e(x))^{-1}} (c\circ\tt\circ\phi^\alpha_\e(x)) \\
   \omega(\alpha,\beta)(e(x)) &=&\left.\frac{\partial^2}{\partial\e\partial r}\right|_{\substack{\e=0\\r=0}}&  \Bigl[
   \Delta^C_{(\phi^\beta_\e(x))^{-1}} \circ \Delta^C_{(\phi^\alpha_r\circ\tt\circ\phi^\beta_\e(x))^{-1}}  
   \circ \Omega_{\phi^\alpha_r\circ\tt\circ\phi^\beta_\e(x),\phi^\beta_\e(x)} (e(x))  \\  
   & & &  -  \Delta^C_{(\phi^\alpha_\e(x))^{-1}}\circ\Delta^C_{(\phi^\beta_r\circ\tt\circ\phi^\alpha_\e(x))^{-1}}  
   \circ\Omega_{\phi^\beta_r\circ\tt\circ\phi^\alpha_\e(x),\phi^\alpha_\e(x)} (e(x)) \Bigr] 
  \end{alignat*}
 where $x\simeq \one_x \in M\subset G$, $e \in \Gamma(E)$, $c\in \Gamma (C)$ and $\alpha,\beta \in \Gamma (A)$ are algebroid sections inducing the right-invariant flows $\phi^\alpha_\e, \phi^\beta_\e$ on $G$.
\end{lemma}
\begin{proof}
 This follows by direct computation using the definitions \eqref{curvaturehorlift} for the connection data out of the Lie algebroid structure which, in turn, comes from the groupoid structure \eqref{eq:splitvbgmaps}  on $\Gamma$. The key point is that the relevant brackets correspond to Lie brackets of the following right-invariant vector fields on $\Gamma$:
 for $c\in \Gamma (C)$, the right-invariant vector field $\hat{c}^R\in \mathcal{X}(\Gamma)$ has flow defined by
 \[ \phi^{\hat{c}^R}_\e (k,g,e) = (k + \e c(\tt(g)),g ,e), \quad (k,g,e)\in \tt^*C\oplus_G {\bold{s}}^*E;\]
 for $\alpha\in \Gamma A$, the right-invariant vector field $\chi_\alpha \in \mathcal{X}(\Gamma)$ defined by the linear section $\alpha^l$ of $\Dd$ is defined by
 \[ \left.\chi_\alpha\right|_{(c,g,e)} = \left.\frac{d}{d\e}\right|_{\e=0} \left(\Delta^C_{\phi^\alpha_\e(x)} c - \Omega_{\phi^\alpha_\e(x), g} (e),\, \phi^\alpha_\e(x) \cdot g,\, e\right) \in T_{(c,g,e)} \Gamma,\]
where $x=\tt(g)$. The formulas then follow straightforwardly from \eqref{curvaturehorlift}.
\end{proof}
\begin{remark}
The above formulas coincide with the ones underlying the differentiation map $\REP(G) \to \REP(A)$ introduced in \cite{AriasSchatz} (when restricted to our particular case of $2$-term representations).
\end{remark}

\begin{example}\label{ex:vbg0_2}
 Let $(\partial,\Delta^E,\Delta^C,\Omega)$ be a representation up to homotopy in which $\partial\equiv 0$ and let $\hat{\omega}\in C^2(G;\Hom(E,C))$ be the associated groupoid $2$-cocycle as in Example \ref{ex:vbg0_1}. Then, the last equation in the above Lemma is equivalent to 
 \[\omega = \VV(\hat{\omega}),\]
 where $\VV:C^\cdot(G;\Hom(E,C))\to \Omega^\cdot(A;\Hom(E,C))$ denotes the Van-Est map \cite{C00}. The  split \vbg $(H_{0,\hat{\omega}},G,E,M)=\VBG(\partial=0,\Delta^E,\Delta^C,\Omega)$ then defines an integration of the underlying type $0$ \vba, namely, 
 \[\LIE(H_{0,\hat{\omega}},G,E,M) \simeq \VBA(\partial=0, \nabla^E,\nabla^C, \VV(\hat{\omega})),\]
 where $\nabla^E$ and $\nabla^C$ are defined by the equations in the above Lemma.
\end{example}

\subsection{Integrability and obstructions}

Finally, the following is the natural notion of integrability for $2$-terms representations up to homotopy of a Lie algebroid stemming from the differentiation map $\Lie$.
\begin{definition}
A representation up to homotopy $(\partial,\nabla^E,\nabla^C,\omega)$ of a Lie algebroid $A$ is said to be {\bf integrable} if there exists an integration $G$ of $A$ and a representation up to homotopy $(\partial_G,\Delta^E,\Delta^C,\Omega) \in \Rep(G)$ such that $\Lie(\partial_G,\Delta^E,\Delta^C,\Omega)$ is isomorphic to $(\partial,\nabla^E,\nabla^C,\omega)$.
\end{definition}

Suppose that $(\partial,\nabla^E,\nabla^C,\omega) \in \Rep(A)$ is integrable in the sense above. Then, an integration $(\partial_G,\Delta^E,\Delta^C,\Omega)\in \Rep(G)$ can be obtained as follows. Take a \vbg $(H,G,E,M)$ integrating the \vba $D=\VBA(\partial,\nabla^E,\nabla^C,\omega)$ and consider a splitting of $H$ (c.f. Remark \ref{rmk:splitG}) producing an element $(\partial_G,\Delta^E,\Delta^C,\Omega)\in \Rep(G)$. Since $\LIE$ preserves isomorphisms, we get:
 \begin{multline*}
  \VBA(\partial,\nabla^E,\nabla^C,\omega)= D\simeq \LIE(H) \simeq \LIE(\VBG(\partial_G,\Delta^E,\Delta^C,\Omega))\\=\VBA(\Lie(\partial_G,\Delta^E,\Delta^C,\Omega)),
   \end{multline*}
 and, thus, that $(\partial_G,\Delta^E,\Delta^C,\Omega)$ integrates $(\partial,\nabla^E,\nabla^C,\omega)$.

\begin{example}\emph{(Adjoint and coadjoint representations)}
Consider an integrable Lie algebroid $A$, and  $G$ an arbitrary Lie groupoid integrating $A$. Since the Lie groupoids $TG$ and $T^*G$ have Lie algebroids which are respectively isomorphic to $TA$ and $T^*\!A$ (see e.g. \cite{mackenzie-book, MX2}), we deduce that both $TA$ and $T^*\!A$ are integrable as \vbas.
The choice of a $TM$-connection on $A$, determines a splitting for both $TA$ and $T^*\!A$ (as in Example \ref{ex:splitTE})
obtaining the so-called adjoint and coadjoint representations up to homotopy $\ad_{A,\!\nabla},$ $\ad^*_{A,\!\nabla} \in \Rep(A)$ of $A$ on the $2$-term complexes $\rho_A:A {\to} TM,$ and $\rho_A^t:T^*M{\to} A^*,$ respectively. 
We thus obtain the fact that $\ad_{A,\nabla}, \ad^*_{A,\!\nabla} \in \Rep(A)$ are integrable if and only if $A$ is integrable. Integrations of $\ad_{A,\nabla}$, $\ad^*_{A,\!\nabla}$ can be obtained by splitting $TG$ and $T^*G$ as in Remark \ref{rmk:splitG}, respectively, for any integration $G$ of $A$.
\end{example}


The integrability problem for a $2$-term representation up to homotopy of a Lie algebroid is then tied to the integrability problem for the corresponding \vba. Here, we shall rephrase the integrability criteria for \vbas given in Section \ref{sec:integrability} in terms of the representation data.

Given a $2$-term representation up to homotopy $(\partial,\nabla^E,\nabla^C,\omega)$ of $A$, and a leaf  a $L\subset M$ of $A$, we obtain a  representation up to homotopy $(\partial_L,\nabla^{E_L},\nabla^{C_L},\omega_L) \in \Rep(A_L)$ by pulling back the superconnection underlying $(\partial,\nabla^E,\nabla^C,\omega)$ along  the inclusion $A_L \hookrightarrow A$. Recall that the associated \vba, namely \break $\VBA(\partial_L,\nabla^{E_L},\nabla^{C_L},\omega_L)$, is of regular type, hence it comes with a class associated to the type $0$ part:
 $$\omega_{L,0} \in H^2\left(A_L, \Hom(\coker \partial_L,\ker\partial_L)\right).$$
 Then, as a direct consequence of Corollary \ref{cor: leaves},  Proposition \ref{prop:reg type} and Theorem \ref{thm:integrals} we obtain the following result.

\begin{corollary}\label{cor:integration reps}
Let $A$ be a Lie algebroid and $(\partial,\nabla^E,\nabla^C,\omega)\in \Rep(A)$ a $2$-term representation up to homotopy of $A$. The following assertions are equivalent:
\begin{enumerate}[i)]
 \item $(\partial,\nabla^E,\nabla^C,\omega)\in \Rep(A)$ is integrable.
\item $A$ is integrable and $(\partial_L,\nabla^{E_L},\nabla^{C_L},\omega_L) \in \Rep(A_L)$ is integrable for each leaf $L$ of $A$.
\item  $A$ is integrable and, for any leaf $L$ of $A$, the periods of $\omega_{L,0}$ vanish:
$$ \int_\sigma \omega_{L,0} = 0, \text{ for every } \sigma \in \pi_2(A_L),\vspace{-5pt}$$
\item  $A$ is integrable and the periods of $\omega$  vanish:
$$\int^{\nabla}_\sigma \omega=0, \text{ for every } \sigma\in \pi_2(A). $$
\end{enumerate}\end{corollary}

\begin{remark}[Relation to simplicial integration]
 In \cite{CA01}, the authors provide an integration scheme that applies to general representations up to homotopy of $A$ on graded vector bundles. 
 In this Remark, we elaborate on how the integrability problem is reflected in their construction.
 Given a Lie algebroid $A$ and a representation up to homotopy $\E \in \REP(A)$, the authors of \cite{CA01} provide a representation up to homotopy ``$\int \E$'' of the (infinite dimensional) infinity groupoid $\Pi_\infty(A)$ associated to $A$. Some of these representations of $\Pi_\infty(A)$ come from representations up to homotopy (in the sense of the present paper) of the Weinstein groupoid $G=\G(A)$ but not all of them do. Indeed, focusing on the $2$-term case and denoting $\E\in\Rep(TS^2)$ the representation up to homotopy underlying Example \ref{ex:non:int}, it is shown in \cite[Prop.\,5.4]{CA01} that $\int \E$ cannot be quasi-isomorphic to a representation comming from $\G(TS^2)$. Within the setting of this paper, this problem can be understood from the fact that the underlying \vba is not integrable, \emph{i.e.} it is an integrability problem for the \vba $D=\VBA(\E)$. 
\end{remark}





\end{document}